\documentclass[12pt,leqno]{article}  % 建檔日期 2014年5月2日
\usepackage{amsmath,amssymb,bbm}
\usepackage{amsthm,array,mathdots}
\usepackage{titlesec} % 重設章節標題所需的巨集
\usepackage[all,cmtip]{xy}
\usepackage{blkarray,arydshln} % 矩陣分隔線所需的巨集

\titlespacing{\section}{0cm}{3.5pc}{1.5pc}

\makeatletter
\def\@citex[#1]#2{\if@filesw\immediate\write\@auxout{\string\citation{#2}}\fi
  \def\@citea{}\@cite{\@for\@citeb:=#2\do
    {\@citea\def\@citea{\@citesep}\@ifundefined
       {b@\@citeb}{{\bf ?}\@warning
       {Citation `\@citeb' on page \thepage \space undefined}}%
{\csname b@\@citeb\endcsname}}}{#1}}
\def\@citesep{; }
\makeatother

\newtheoremstyle{Kang}{}{}{\itshape}{}{\bf}{}{.5em}{}
\theoremstyle{Kang}
\newtheorem{theorem}{Theorem}[section]

\newtheorem{lemma}[theorem]{Lemma}

\newtheorem{prop}[theorem]{Proposition}

\newtheoremstyle{Kremark}{}{}{}{}{\bf}{}{.5em}{}
\theoremstyle{Kremark}
\newtheorem*{remark}{Remark.}
\newtheorem{defn}[theorem]{Definition}

\newtheorem{other}{}

\newenvironment{Case}[1]{\medskip {\it Case #1.}}{}

\allowdisplaybreaks[1]  % 允許多行數學式分頁

\def\fn#1{\operatorname{#1}} % function work like \sin
\def\bm#1{\mathbbm{#1}}
\def\c#1{\mathcal{#1}}

\title{Function Fields of Algebraic Tori Revisited}

\author{\begin{minipage}{0.4\textwidth}
Shizuo Endo \\[2mm] \normalsize
Department of Mathematics \\
Tokyo Metropolitan University \\
Tokyo, Japan \\
E-mail: wktxj285@yahoo.co.jp
\end{minipage}
and \
\begin{minipage}{0.4\textwidth}
Ming-chang Kang \\[2mm] \normalsize
Department of Mathematics \\
National Taiwan University \\
Taipei, Taiwan \\
E-mail: kang@math.ntu.edu.tw
\end{minipage} }
\date{}

\begin{document}

\maketitle

\footnote{\textit{\!\!\! $2010$ Mathematics Subject
Classification}. 14E08, 11R33, 20C10, 11R29.}
\footnote{\textit{\!\!\! Keywords and phrases}. Algebraic torus,
rationality problem, locally free class groups, class numbers,
maximal orders, twisted group rings.}

\begin{abstract}
{\noindent\bf Abstract.} Let $K/k$ be a finite Galois extension
and $\pi = \fn{Gal}(K/k)$. An algebraic torus $T$ defined over $k$
is called a $\pi$-torus if $T\times_{\fn{Spec}(k)}
\fn{Spec}(K)\simeq \bm{G}_{m,K}^n$ for some integer $n$. The set
of all algebraic $\pi$-tori defined over $k$ under the stably birational
equivalence forms a semigroup, denoted by $T(\pi)$. We will give a
complete proof of the following theorem due to Endo and Miyata
\cite{EM4}. Theorem. Let $\pi$ be a finite group. Then
$T(\pi)\simeq C(\Omega_{\bm{Z}\pi})$ where $\Omega_{\bm{Z}\pi}$ is
a maximal $\bm{Z}$-order in $\bm{Q}\pi$ containing $\bm{Z}\pi$ and
$C(\Omega_{\bm{Z}\pi})$ is the locally free class group of
$\Omega_{\bm{Z}\pi}$, provided that $\pi$ is isomorphic to one of the
following four types of groups : $C_n$ ($n$ is any positive
integer), $D_m$ ($m$ is any odd integer $\ge 3$), $C_{q^f}\times
D_m$ ($m$ is any odd integer $\ge 3$, $q$ is an odd prime number
not dividing $m$, $f\ge 1$, and
$(\bm{Z}/q^f\bm{Z})^{\times}=\langle \bar{p}\rangle$ for any prime
divisor $p$ of $m$), $Q_{4m}$ ($m$ is any odd integer $\ge 3$,
$p\equiv 3 \pmod{4}$ for any prime divisor $p$ of $m$).
\end{abstract}

\newpage
%------------------------------------S1
\section{Introduction}

In \cite{EM4}, Endo and Miyata investigated the classification of
the function fields of algebraic tori. An additional paper was
planned, which would contain a complete proof of $(1')\Rightarrow
(2)$ of Theorem 3.3 (see \cite[page 187]{EM4}). This article was
announced in \cite[p.189, line $-14$]{EM4}. Unfortunately the plan
didn't materialize. The present article may be regarded as a
supplement to the papers \cite{EM4} and the paper \cite{Sw6}. We
thank Prof. Richard G. Swan who detected a mistake in
\cite[p.96]{EM4} and in the first version of this paper (see the
remark at the end of Section 4). He also showed us how we could
simplify the proof of Theorem \ref{t4.3} in a revised version of
the first version.

To begin with, we recall some definitions and terminology.
Let $k$ be a field, $k\subset L$ be a field extension.
The field $L$ is said to be rational over $k$ (in short, $k$-rational) if, for some $n$,
$L\simeq k(X_1,\ldots,X_n)$ over $k$ where $k(X_1,\ldots,X_n)$ is the rational function field of $n$ variables over $k$.
Two field extensions $k\subset L_1, L_2$ are stably isomorphic over $k$ if,
$L_1(X_1,\ldots,X_m)\simeq L_2(Y_1,\ldots,Y_n)$ over $k$ where $X_1,\ldots,X_m$ are algebraically independent over $L_1$ and $Y_1,\ldots,Y_n$ are algebraically independent over $L_2$.
In particular,
a field extension $k\subset L$ is stably $k$-rational if $L(X_1,\ldots,X_m)$ is $k$-rational where $X_1,\ldots,X_m$ are some elements algebraically independent over $L$. When $k$ is an infinite field, a field extension $L$ over $k$ is said to be retract
$k$-rational if there is a $k$-algebra $A$ contained in $L$ such
that (i) $L$ is the quotient field of $A$, (ii) there exist a
non-zero polynomial $f\in k[X_1,\ldots,X_n]$ (where
$k[X_1,\ldots,X_n]$ is the polynomial ring) and $k$-algebra
morphisms $\varphi\colon A\to k[X_1,\ldots,X_n][1/f]$ and
$\psi\colon k[X_1,\ldots,X_n][1/f]\to A$ satisfying
$\psi\circ\varphi =1_A$ (see \cite{Sa}).
When $V$ is an irreducible algebraic variety defined over $k$,
$V$ is $k$-rational (resp.\ stably $k$-rational, retract $k$-rational) if so is the function field $k(V)$ over $k$.
If $V_1$ and $V_2$ are irreducible varieties over $k$,
$V_1$ and $V_2$ are stably birational equivalent over $k$ if so are their function fields over $k$.

An algebraic torus $T$ defined over a field $k$ is an affine
algebraic group defined over $k$ such that $T\times_{\fn{Spec}(k)}
\fn{Spec}(\bar{k})\simeq \bm{G}_{m,\bar{k}}^n$ for some integer
$n$ where $\bar{k}$ is the algebraic closure of $k$ and
$\bm{G}_{m,K}$ is the 1-dimensional multiplicative group defined
over a field $K$ (containing the base field $k$) \cite[page 36;
Vo]{On,Sw5}. By \cite[Proposition 1.2.1]{On}, for any algebraic
torus $T$ over $k$, there is a finite separable extension field
$K$ of $k$ satisfying that $T\times_{\fn{Spec}(k)}
\fn{Spec}(K)\simeq \bm{G}_{m,K}^n$; such a field $K$ is called a
splitting field of $T$.

Let $k$ be a field, $\pi$ be a finite group. We will say that the
field $k$ admits a $\pi$-extension if there is a Galois field
extension $K/k$ such that $\pi = \fn{Gal}(K/k)$.

%-----------------d1.1
\begin{defn} \label{d1.1}
Let $\pi$ be a finite group, $k$ be a field admitting a
$\pi$-extension. An algebraic torus $T$ over $k$ is called a
$\pi$-torus if it has a splitting field $K$ which is Galois over
$k$ with $\fn{Gal}(K/k) = \pi$.
\end{defn}

Let $\pi$ be a finite group.
Recall that a finitely generated $\bm{Z}\pi$-module $M$ is called a $\pi$-lattice if it is torsion-free as an abelian group.

If $T$ is a $\pi$-torus over a field $k$ with $\pi=\fn{Gal}(K/k)$,
then its character module $\fn{Hom}(T\times_{\fn{Spec}(k)}
\fn{Spec}(K),\bm{G}_{m,K})$ is a $\pi$-lattice. Conversely, every
$\pi$-lattice $M$ is isomorphic to the character module of some
algebraic $\pi$-torus $T$ over $k$ (as $\bm{Z}\pi$-modules)
\cite[page 36]{On,Sw5}.

%-------------------d1.2
\begin{defn} \label{d1.2}
Let $K/k$ be a finite Galois field extension with $\pi=\fn{Gal}(K/k)$.
Let $M=\bigoplus_{1\le i\le n} \bm{Z}\cdot e_i$ be a $\pi$-lattice.
We define an action of $\pi$ on $K(M)=K(x_1,\ldots,x_n)$,
the rational function field of $n$ variables over $K$,
by $\sigma\cdot x_j=\prod_{1\le i\le n} x_i^{a_{ij}}$ if $\sigma\cdot e_j=\sum_{1\le i\le n} a_{ij} e_i \in M$,
for any $\sigma \in \pi$ (note that $\pi$ acts on $K$ also).
The fixed field is denoted by $K(M)^\pi$.
\end{defn}

Let $K/k$ be a finite Galois extension with $\pi=\fn{Gal}(K/k)$.
There is a duality between the category of algebraic $\pi$-tori
defined over $k$ and the category of $\pi$-lattices. In fact, if
$T$ is a $\pi$-torus and $M$ is its character module, then the
function field of $T$ is isomorphic to $K(M)^\pi$ (see \cite[page
36]{On,Sw5}). Thus the study of rationality problems of $\pi$-tori
is reduced to that of $\pi$-lattices.

%--------------------d1.3
\begin{defn}[{\cite[page 86]{EM1,EM4}}] \label{d1.3}
Let $\pi = \fn{Gal}(K/k)$ be a finite group where $K/k$ is a
Galois extension. Define an equivalence relation in the category
of $\pi$-lattices: Two $\pi$-lattices $M$ and $N$ are equivalent,
denoted by $M-N$, if the fields $K(M)^\pi$ and $K(N)^\pi$ are
stably isomorphic over $k$, i.e.\ $K(M)^\pi (X_1,\ldots,X_m)\simeq
K(N)^\pi(Y_1,\ldots,Y_n)$ for some algebraically independent
elements $X_i$, $Y_j$.

Let $\pi$ be a finite group. Define a commutative monoid $T(\pi)$
as follows. As a set, $T(\pi)$ is the set of all equivalence
classes $[M]$ under the equivalence relation ``$-$" defined above
(note that $[M]$ is the equivalence class containing the
$\pi$-lattice $M$); the monoid operation is defined by $[M]+[N]
=[M \oplus N]$.
\end{defn}

%--------------------t1.4
\begin{theorem}[{\cite[page 95, Theorem 3.3; page 187]{EM4}}] \label{t1.4}
Let $\pi$ be a finite group.
Then the following statements are equivalent:
\begin{enumerate}
\item[{\rm (1)}]
$\pi$ is isomorphic to
\begin{enumerate}
\item[{\rm (i)}] a cyclic group $C_n$ where $n$ is any positive
integer, or \item[{\rm (ii)}] a dihedral group $D_m$ of order $2m$
where $m$ is an odd integer $\ge 3$, or \item[{\rm (iii)}] a
direct product $C_{q^f}\times D_m$ where $q$ is an odd prime
number, $f\ge 1$, $m$ an odd integer $\ge 3$, $\gcd\{q,m\}=1$ and
for any prime divisor $p$ of $m$, $(\bm{Z}/q^f
\bm{Z})^{\times}=\langle\bar{p}\rangle$, or \item[{\rm (iv)}]
$Q_{4m}=\langle\sigma,\tau:
\sigma^{2m}=\tau^4=1,\sigma^m=\tau^2,\tau^{-1}\sigma\tau=\sigma^{-1}\rangle$,
the generalized quaternion group of order $4m$, where $m\ge 3$ is
an odd integer and $p\equiv 3 \pmod{4}$ for any prime divisor $p$
of $m$.
\end{enumerate}
\item[{\rm (2)}]
$T(\pi)\simeq C(\bm{Z}\pi)/C^q(\bm{Z}\pi) \simeq C(\Omega_{\bm{Z}\pi})$ where $\Omega_{\bm{Z}\pi}$ is a maximal $\bm{Z}$-order in $\bm{Q}\pi$ containing $\bm{Z}\pi$.
\item[{\rm (3)}]
$T(\pi)$ is a finite group.
\end{enumerate}
\end{theorem}

The purpose of this article is to give a proof of Theorem
\ref{t1.4} supplementing the proof outlined in \cite{EM4}. Note
that the definition of the locally free class groups
$C(\bm{Z}\pi)$, $C(\Omega_{\bm{Z}\pi})$ and the associated group
$C^q(\bm{Z}\pi)$ may be found in Definition \ref{d2.11} and
Definition \ref{d2.12}; the isomorphisms $T(\pi)\simeq
C(\bm{Z}\pi)/C^q(\bm{Z}\pi)$ and $C(\bm{Z}\pi)/C^q(\bm{Z}\pi)
\simeq C(\Omega_{\bm{Z}\pi})$ are described in Definition
\ref{d2.13}.

The main ideas of the proof of $(1)\Rightarrow(2)$ in Theorem
\ref{t1.4} will be explained at the beginning of Section 5.

The organization of this paper is as follows. In Section 2, the
notions of flabby, coflabby, invertible, permutation
$\pi$-lattices are recalled. Some fundamental results are
summarized also. Section 3 contains the definitions of twisted
group rings, denoted by $S\circ G$, and crossed-product orders,
denoted by $(S\circ G)_f$ (where $f$ is a $2$-cocycle of $G$).
Some results of Rosen's Ph.D. dissertation (unpublished in the
journals) will be quoted for the convenience of the reader. We suggest the reader to consult similar results in Lee's paper \cite{Lee}.
Section 4 is a first step of the proof of $(1)\Rightarrow (2)$ of
Theorem \ref{t1.4}; this section contains a devissage theorem for
$[M]^{fl}$ where $M$ is an invertible $\pi$-lattice (for the definition of $[M]^{fl}$, see Definition \ref{d2.1}). Section 5 is
devoted to the proof of $(1)\Rightarrow (2)$ of Theorem
\ref{t1.4}. In the final section, we compute
$C(\Omega_{\bm{Z}\pi})$ when $\pi$ are the groups in Theorem
\ref{t1.4}. As a consequence, we deduce a result about
$D_{p^c}$-tori; the case of $D_p$-tori was proved by Hoshi, Kang
and Yamasaki  by a different method \cite{HKY}. On the other hand,
some properties of $\pi$-lattices proved by Colliot-Th\'el\`ene and
Sansuc, when $\pi$ is the Klein four-group or the quaternion group
of order $8$ \cite[pages 186-187]{CTS}, will be generalized to the
situation when $\pi$ is some dihedral $2$-group or quaternion
groups of order 8, 16, 32, etc.; see Proposition \ref{p6.8} and
Proposition \ref{p6.9}.

\bigskip
Terminology and notations.
Throughout the paper, we denote by $k$ a field and by $\pi$ a finite group.
We denote by $\bm{Z}\pi$ the integral group ring of the group $\pi$.

Denote by $C_n$ (resp.\ $D_n$) the cyclic group of order $n$
(resp.\ the dihedral group of order $2n$). The group $Q_{4n}$
denotes the generalized quaternion group of order $4n$ where $n\ge
2$, i.e.\
$Q_{4n}=\langle\sigma,\tau:\sigma^{2n}=\tau^4=1,\sigma^n=\tau^2,\tau^{-1}\sigma\tau=\sigma^{-1}\rangle$.

If $q$ is a prime power, $\bm{F}_q$ denotes the finite field with
$q$ elements. For any positive integer $n\ge 2$, $\zeta_n$ denotes
a primitive $n$-th root of unity and $\Phi_n(X)\in\bm{Z}[X]$ the
$n$-th cyclotomic polynomial. $(\bm{Z}/n\bm{Z})^{\times}$ is the
group of units of the ring $\bm{Z}/n\bm{Z}$. A commutative integral domain $R$ is called a DVR if it is a
discrete rank-one valuation ring. If $R$ is a ring, $M_n(R)$
denotes the matrix ring of all $n \times n$ matrices over $R$. If $M$ is an $A$-module where $A$ is a ring, we denote
by $M^{(n)}$ the direct sum of $n$ copies of the module $M$.

When $\pi$ is a finite group and $\Lambda$ is a $\bm{Z}$-order satisfying that
$\bm{Z}\pi \subset \Lambda \subset \bm{Q}\pi$ (see \cite[page 524,
Definition 23.2]{CR1} for the definition of an order), a locally free $\Lambda$-lattice of rank $n$ is a finitely generated $\Lambda$-module $M$ such that $M_P \simeq (\bm{Z}_P \otimes_{\bm{Z}} \Lambda)^{(n)}$ for any non-zero prime ideal $P$ of $\bm{Z}$ where $\bm{Z}_P$ is the localization of $\bm{Z}$ at the prime ideal $P$ and $M_P=\bm{Z}_P \otimes_{\bm{Z}}M$ \cite[page 382, Defintion 55.28]{CR2}. A projective ideal over $\bm{Z}\pi$ is a left ideal $\c{A}$ of $\bm{Z}\pi$  such that $\c{A}$ is also a projective $\bm{Z}\pi$-module. By \cite{Sw1}, $\c{A}$ is a projective ideal over $\bm{Z}\pi$ if and only if it is a locally free $\bm{Z}\pi$-lattice of rank one.

We remind the reader that the definitions of $T(\pi)$, $T^g(\pi)$,
$C(\Lambda)$ and $[M]^{fl}$, are given in Definition \ref{d1.3},
Definition \ref{d2.8}, Definition \ref{d2.11} and Definition
\ref{d2.1}; the monoid $F_\pi$ is defined in the paragraph before
Definition \ref{d2.1}.

%------------------------------S2
\section{Preliminaries}

In this section we recall some notions related to $\pi$-lattices
and summarize several results in \cite{EM4} and \cite{Sw5}.

A $\pi$-lattice $M$ is called a permutation lattice if $M$ has a
$\bm{Z}$-basis permuted by $\pi$; $M$ is called an invertible
lattice if it is a direct summand of some permutation lattice. A
$\pi$-lattice $M$ is called a flabby lattice (or a flasque lattice) if $H^{-1}(\pi',M)=0$
for any subgroup $\pi'$ of $\pi$; it is called coflabby (or a coflasque lattice) if
$H^1(\pi',M)=0$ for any subgroup $\pi'$ of $\pi$. For details, see
\cite{EM4,CTS,Sw5}.

Two $\pi$-lattices $M_1$ and $M_2$ are similar, denoted by
$M_1\sim M_2$, if $M_1\oplus P_1\simeq M_2\oplus P_2$ for some
permutation $\pi$-lattices $P_1$ and $P_2$. The flabby class
monoid $F_\pi$ is the monoid whose elements consist of flabby
$\pi$-lattices under the above similarity relation. A typical
element in $F_\pi$ is $[M]$, the equivalence class containing $M$
where $M$ is a flabby $\pi$-lattice; the monoid operation on
$F_\pi$ is defined by $[M_1]+[M_2]=[M_1\oplus M_2]$. Thus $F_\pi$
becomes an abelian monoid and $[P]$ is the identity element of it
where $P$ is any permutation $\pi$-lattice \cite[page 33]{Sw5}.

%-----------------d2.1
\begin{defn} \label{d2.1}
Let $\pi$ be a finite group, $M$ be any $\pi$-lattice. Then $M$
has a flabby resolution, i.e.\ there is an exact sequence of
$\pi$-lattices: $0\to M\to P\to E\to 0$ where $P$ is a permutation
lattice and $E$ is a flabby lattice \cite[Lemma 1.1]{EM4}. The
class $[E]\in F_\pi$ is uniquely determined by the lattice $M$
\cite[Lemma 8.7]{Sw5}. We define $[M]^{fl}=[E]\in F_\pi$.
Sometimes we will say that $[M]^{fl}$ is permutation or invertible
if the class $[E]$ contains a permutation or invertible lattice.
\end{defn}

Be aware that the equivalence relation $M-N$ in Definition
\ref{d1.3} is different from the above similarity relation
$M_1\sim M_2$. The monoids $T(\pi)$ in Definition \ref{d1.3} and
the above monoid $F_\pi$ are isomorphic through the following
lemma.

%--------------------l2.2
\begin{lemma} \label{l2.2}
Let $\pi$ be a finite group.

{\rm (1)} If $M$ and $N$ are $\pi$-lattices.
Then $M-N$ if and only if $[M]^{fl}=[N]^{fl}$.

{\rm (2)} Define a monoid homomorphism $\Phi: T(\pi)\to F_\pi$ by $\Phi([M])=[M]^{fl}$.
Then $\Phi$ is an isomorphism.
\end{lemma}

\begin{proof}
Let $\pi=\fn{Gal}(K/k)$. By the same idea in the proof of
\cite[Theorem 1.7]{Le}, it is not difficult to show that
$K(M)^\pi$ and $K(N)^\pi$ are stably isomorphic over $k$ if and
only if there exist exact sequences of $\pi$-lattices $0\to M\to
E\to P\to 0$ and $0\to N\to E\to Q\to 0$ where $E$ is some
$\pi$-lattice and $P$, $Q$ are permutation $\pi$-lattices. The
latter condition is equivalent to $[M]^{fl}=[N]^{fl}$ by
\cite[Lemma 8.8]{Sw5}.

For the isomorphism of $\Phi$, note that $\Phi$ is well-defined by
(1). If $\Phi([M])=0$, then $[M]^{fl}=0$. Thus $K(M)^\pi$ is
stably $k$-rational by the following Lemma \ref{l2.4}. Hence
$M-\bm{Z}$ where $\bm{Z}$ is the  $\pi$-lattice with trivial $\pi$
actions. Thus $\Phi$ is injective. On the other hand, if $E$ is
any flabby $\pi$-lattice, let $E^0=\fn{Hom}_{\bm{Z}} (E,\bm{Z})$
be its dual lattice. Take a flabby resolution of $E^0$, $0\to
E^0\to P\to F\to 0$ as in Definition \ref{d2.1}. We get an exact
sequence $0\to F^0\to P^0\to E\to 0$. Thus
$[E]=[F^0]^{fl}=\Phi([F^0])$ and $\Phi$ is surjective.
\end{proof}

The following lemma is due to  Endo and Miyata \cite[Theorem
1.2]{EM2}, Voskresenskii and Saltman.
%--------------------l2.3
 \begin{lemma} {\rm (e.g. {\cite[Theorem 1.7; Sa, Theorem 3.14]{Le}})} \label{l2.4} Let
$K/k$ be a finite Galois field extension, $\pi=\fn{Gal}(K/k)$, $M$
be a $\pi$-lattice. Then {\rm (i)} $K(M)^\pi$ is stably $k$-rational if and
only if $[M]^{fl}=0$ in $F_\pi$, and {\rm (ii)} $K(M)^\pi$ is retract $k$-rational if and
only if $[M]^{fl}$ is an invertible lattice.
\end{lemma}

%--------------------l2.4
\begin{lemma} \label{l2.5}
Let $\pi$ be a finite group, and let $0\to M'\to M\to M''\to 0$ be an exact sequence of $\pi$-lattices.

{\rm (1) (e.g. \cite[Proposition 1.5]{Le})} If $M''$ is
invertible, then $[M]^{fl}=[M']^{fl}+[M'']^{fl}$.

{\rm (2) (\cite[Lemma 2.2]{EM4})} If all the Sylow subgroups of
$\pi$ are cyclic and $\widehat{H}^0 (\pi',M')=0$ for any subgroup
$\pi'\subset\pi$ , then $[M]^{fl}=[M']^{fl}+[M'']^{fl}$.
\rm{(}Note that $\widehat{H}^0 (\pi',M')$ denotes the Tate
cohomology.\rm{)}
\end{lemma}

%--------------------l2.5
\begin{lemma} \label{l2.6}
Let $\pi$ be a finite group, $M$ be a $\pi$-lattice.

{\rm (1) (e.g. \cite[Proposition 1.2]{Le})} $M$ is invertible if
and only if, for any coflabby $\pi$-lattice $C$, any short exact
sequence $0\to C\to E\to M\to 0$ splits.

{\rm (2) (\cite[Lemma 1.1]{EM6})} There is a short exact sequence
of $\pi$-lattices $0\to M\to C\to P\to 0$ such that $C$ is
coflabby and $P$ is permutation.
\end{lemma}

%---------------------t2.6
\begin{theorem}[{\cite[Theorem 1.5]{EM4}}] \label{t2.7}
Let $\pi$ be a finite group. Then all the flabby $\pi$-lattices
(resp.\ all the coflabby $\pi$-lattices) are invertible if and
only if all the Sylow subgroups of $\pi$ are cyclic.
\end{theorem}

In the literature there are two different definitions for metacyclic groups. A finite group is metacyclic if it is an extension of a cyclic group by another cyclic group (see \cite[page 160]{Is}). Thus we will call a metacyclic group $\pi$ a split metacyclic group if $\pi$ has a cyclic normal subgroup $\pi_0$ such that $\pi/\pi_0$ is cyclic and $\gcd\{|\pi_0|,|\pi/\pi_0|\}=1$. It is known that $\pi$ is split metacyclic if and only if all the Sylow subgroups of $\pi$ are cyclic \cite[page 160, Theorem 5.16]{Is}. However, in Zassenhaus's monograph \cite[page 174]{Za}, a finite group $\pi$ is called metacyclic group if $\pi'$ and $\pi/\pi'$ are cyclic groups where $\pi'$ is the commutator subgroup of $\pi$ (see \cite[page 175, Theorem 11]{Za} also).

\medskip
\begin{theorem}[{\cite[Theorem 2.1]{EM6}}] \label{t2.15}
Let $\pi$ be a finite group. Then the $\pi$-lattices which are
both flabby and coflabby are necessarily invertible if and only if
all the $p$-Sylow subgroups of $\pi$ are cyclic for odd prime $p$,
and all the $2$-Sylow subgroups of $\pi$ are cyclic or dihedral
(including the Klein four-group).
\end{theorem}

%----------------------d2.7
\begin{defn}[{\cite[page 86]{EM4}}] \label{d2.8}
Let $\pi$ be a finite group.
Define $T^g(\pi)=\{[M]\in T(\pi):$ There exists $[N]\in T(\pi)$ such that $[M]+[N]=0$ in $T(\pi)\}$.
It follows that $T^g(\pi)$ is a subgroup of $T(\pi)$;
it is the maximal subgroup of $T(\pi)$.
By Jacobinski's Theorem \cite[Proposition 5.8]{Ja},
$T^g(\pi)$ is finitely generated.

If $M$ is a $\pi$-lattice and $[M]\in T^g(\pi)$,
then there is an invertible $\pi$-lattice $E$ such that $[M]=[E]$ in $T^g(\pi)$ (see \cite[Lemma 1.6]{EM4}).
\end{defn}

Now we turn to $R$-orders in a separable $K$-algebra $\Sigma$
where $R$ is a Dedekind domain with quotient field $K$ \cite[page
523]{CR1}. Recall that a separable $K$-algebra $\Sigma$ is a
finite-dimensional semi-simple algebra over $K$ such that the
center $K'$ of $\Sigma$ is an \'etale $K$-algebra, i.e.\
$K'=\prod_{1\le i\le t} K_i$ where each $K_i$ is a finite
separable field extension of $K$. An $R$-order $\Lambda$ is a
subring of $\Sigma$, $R\subset \Lambda\subset \Sigma$ satisfying
that $K\Lambda=\Sigma$ and $\Lambda$ is a finitely generated
$R$-module.

%----------------------d2.8
\begin{defn} \label{d2.9}
Let $R$ be a Dedekind domain with quotient field $K$, $\Lambda$ be
an $R$-order in a separable $K$-algebra $\Sigma$. A
$\Lambda$-lattice $M$ is a left $\Lambda$-module which is finitely
generated and projective as an $R$-module \cite[page 524]{CR1}.
Two $\Lambda$-lattices $M$ and $N$ are in the same genus if
$M_P\simeq N_P$ for any prime ideal $P$ of $R$ where
$M_P=R_P\otimes_R M$ and $R_P$ is the localization of $R$ at the
prime ideal $P$ \cite[pages 642--643]{CR1}.
\end{defn}

%---------------------t2.9
\begin{theorem}[Jacobinski, Roiter {\cite[page 660, Theorem 31.28]{CR2}}] \label{t2.10}
Let $R$ be a Dedekind domain whose quotient field is a global
field $K$. Let $\Sigma$ be a separable $K$-algebra and $\Lambda$
be an $R$-order in $\Sigma$. Let $M$ and $N$ be $\Lambda$-lattices
in the same genus and $F$ be a faithful $\Lambda$-lattice. Then
there is a $\Lambda$-lattice $F'$ such that $F$, $F'$ are in the
same genus and $M\oplus F\simeq N\oplus F'$. In particular, if $M$
and $N$ are $\Lambda$-lattices in the same genus, then $M\oplus
\Lambda \simeq N\oplus \c{A}$ for some projective ideal $\c{A}$
over $\Lambda$.
\end{theorem}

%--------------------d2.10
\begin{defn} \label{d2.11}
Let $\pi$ be a finite group, and let $\Lambda$ be a $\bm{Z}$-order
with $\bm{Z}\pi \subset \Lambda\subset \bm{Q}\pi$. Recall the definition of locally free
$\Lambda$-modules given at the end of Section 1. We define
$K_0^{lf}(\Lambda)$, the Grothendieck group of the category of
locally free $\Lambda$-modules, as follows: $K_0^{lf}(\Lambda)$ is the abelian group with generators $[P]$ for locally free $\Lambda$-modules $[P]$ and relations $[P]=[P']+[P'']$ if $0 \to P' \to P \to P'' \to 0$ is an exact sequence of locally free $\Lambda$-modules. Taking ranks provides a surjective group homomorphism $\varphi:K_0^{lf}(\Lambda) \to \bm{Z}$. The kernel of $\varphi$ is called the locally
free class group of $\Lambda$, and is denoted by $C(\Lambda)$. The group $C(\Lambda)$ is also
called the projective class group of $\Lambda$.

When $\Lambda=\bm{Z}\pi$, the locally free class group $C(\bm{Z}\pi)$ may be defined via $K_0(\bm{Z}\pi)$, the usual Grothendieck group of the category of
finitely generated projective $\bm{Z}\pi$-modules: $C(\bm{Z}\pi)$ is a subgroup of $K_0(\bm{Z}\pi)$ defined as
$C(\bm{Z}\pi)=\{[\c{A}]-[\bm{Z}\pi]\in K_0(\bm{Z}\pi):\c{A}$ is a
projective ideal over $\bm{Z}\pi\}$, because a locally free $\bm{Z}\pi$-module of rank $n$ is isomorphic to a direct sum of $\bm{Z}\pi^{(n-1)}$ and some projective ideal over $\bm{Z}\pi$ by \cite{Sw1}.

Note that $C(\Lambda)$ is a finite group \cite[page 51,
Proposition 39.13]{CR2}. For details of $C(\Lambda)$, see
\cite[page 397; CR2, page 50, page 219, page 230]{EM2}.
\end{defn}

%---------------------d2.11
\begin{defn} \label{d2.12}
Let $\pi$ be a finite group and $\Omega_{\bm{Z}\pi}$ be a maximal
$\bm{Z}$-order in $\bm{Q}\pi$ satisfying $\bm{Z}\pi \subset
\Omega_{\bm{Z}\pi}\subset \bm{Q}\pi$. It is known that the natural
map $\varphi_1:C(\bm{Z}\pi)\to C(\Omega_{\bm{Z}\pi})$ defined by
$\varphi_1([\c{A}]-[\bm{Z}\pi])=[\Omega_{\bm{Z}\pi}
\otimes_{\bm{Z}\pi} \c{A}]-[\Omega_{\bm{Z}\pi}]$ is surjective
(see \cite[page 290, Theorem 49.25]{CR2}). Define
$\widetilde{C}(\bm{Z}\pi)=\fn{Ker}(\varphi_1)$. From the
definition, $\widetilde{C}(\bm{Z}\pi)=\{[\c{A}]-[\bm{Z}\pi]\in
C(\bm{Z}\pi):
(\Omega_{\bm{Z}\pi}\otimes_{\bm{Z}\pi}\c{A})\oplus\Omega_{\bm{Z}\pi}\simeq\Omega_{\bm{Z}\pi}
\oplus\Omega_{\bm{Z}\pi}\}$ (see \cite{EM4}).

In \cite[page 234]{CR2} and other literature,
$\widetilde{C}(\bm{Z}\pi)$ is written as $D(\bm{Z}\pi)$; but we
choose sticking to the notation in \cite{EM2,EM4}.

On the other hand, a projective ideal $\c{A}$ over $\bm{Z}\pi$ may
be regarded as a $\pi$-lattice. Thus we may define
$\varphi_2:C(\bm{Z}\pi)\to T^g(\pi)$ defined by
$\varphi_2([\c{A}]-[\bm{Z}\pi])=[\c{A}]\in T^g(\pi)$. It is easy
to check that $\varphi_2$ is a well-defined morphism of abelian
groups. Define $C^q(\bm{Z}\pi)=\fn{Ker}(\varphi_2)$. Clearly
$C^q(\bm{Z}\pi)=\{[\c{A}]-[\bm{Z}\pi]\in C(\bm{Z}\pi):\c{A}$ is a
projective ideal over $\bm{Z}\pi$ satisfying $[\c{A}]^{fl}=0$ in
$F_\pi\}$.

Finally define
$\widetilde{C}^q(\bm{Z}\pi)=\{[\c{A}]-[\bm{Z}\pi]\in C^q
(\bm{Z}\pi):\c{A}$ is a projective ideal over $\bm{Z}\pi$
satisfying $\c{A}\oplus P\simeq \bm{Z}\pi \oplus P$ for some
permutation $\pi$-lattice $P\}$ (see \cite[page 698]{EM2}).

It can be shown that $\widetilde{C}^q(\bm{Z}\pi)$ is a subgroup of
$\widetilde{C}(\bm{Z}\pi)$. In fact, Oliver proves that
$\widetilde{C}^q(\bm{Z}\pi)=\widetilde{C}(\bm{Z}\pi)$ for any
finite group $\pi$ (see \cite[Theorem 5]{Ol}). Hence
$\widetilde{C}(\bm{Z}\pi)\subset C^q(\bm{Z}\pi)$. It follows that
there is always a surjective map $C(\Omega_{\bm{Z}\pi})\to
C(\bm{Z}\pi)/C^q(\bm{Z}\pi)$.

We remark that there exists a meta-cyclic group $\pi$ such that
$\widetilde{C}(\bm{Z}\pi) \neq C^q(\bm{Z}\pi)$ (\cite[page 709,
Example 4.3]{EM2}). However, it is shown \cite{EM2} that
$\widetilde{C}(\bm{Z}\pi)=C^q(\bm{Z}\pi)$ for many groups $\pi$,
including those groups listed in Theorem \ref{t1.4} (1) (see
Section 5).
\end{defn}

%---------------------------------d2.12
\begin{defn} \label{d2.13}
Let $\pi$ be a finite group satisfying $C^q (\bm{Z}\pi)=\widetilde{C}(\bm{Z}\pi)$.

Let $\varphi_1$ and $\varphi_2$ be the group homomorphisms in
Definition \ref{d2.12}. Since $0\to C^q(\bm{Z}\pi)\to C(\bm{Z}\pi)
\xrightarrow{\varphi_2} T^g(\pi)$ is left-exact, we will write
$T^g(\pi)\simeq C(\bm{Z}\pi)/C^q(\bm{Z}\pi)$ when the canonical
injection
$C(\bm{Z}\pi)/C^q(\bm{Z}\pi)\xrightarrow{\bar{\varphi}_2}
T^g(\pi)$ is an isomorphism.

Similarly, when $C^q(\bm{Z}\pi)=\widetilde{C}(\bm{Z}\pi)$ and
$c:T^g(\pi)\to T(\pi)$ is the inclusion map,
then the composite map $c\cdot \bar{\varphi}_2\cdot \bar{\varphi}_1^{-1}: C(\Omega_{\bm{Z}\pi})\to C(\bm{Z}\pi)/\widetilde{C}(\bm{Z}\pi) %
=C(\bm{Z}\pi)/C^q(\bm{Z}\pi) \to T^g(\pi)\to T(\pi)$ is
well-defined. We will write $T(\pi)\simeq C(\Omega_{\bm{Z}\pi})$
if the injection $c\cdot \bar{\varphi}_2\cdot
\bar{\varphi}_1^{-1}$ is surjective and hence is an isomorphism.
\end{defn}

%----------------------------t2.13
\begin{theorem} \label{t2.14}
Let $\pi$ be a finite group.

{\rm (1) (\cite[Proposition 3.1]{EM4})} $T^g(\pi)\simeq
C(\bm{Z}\pi)/C^q(\bm{Z}\pi)$ if and only if, for any invertible
$\pi$-lattice $M$, there is a projective ideal $\c{A}$ over
$\bm{Z}\pi$ such that $[M]^{fl}=[\c{A}]^{fl}$.

{\rm (2) (\cite[Theorem 3.2]{EM4})} If $\pi$ is a $p$-group, then
$T^g(\pi)\simeq C(\bm{Z}\pi)/C^q(\bm{Z}\pi)$.

\end{theorem}

%-----------------------------------------------S3
\section{Twisted group rings}

%------------------------d3.1
\begin{defn}[{\cite[page 183, page 599--600; CR2, page 291]{CR1}}] \label{d3.1}
Let $K$ be a field, $L/K$ be a finite Galois field extension with
$G=\fn{Gal}(L/K)$, and $f:G\times G\to L^{\times}$ be a
$2$-cocycle of $G$ where $L^{\times}=L\backslash \{0\}$ is the
multiplicative group of $L$. The crossed product algebra, denoted
by $(L\circ G)_f$, is defined by
\[
(L\circ G)_f=\bigoplus_{\sigma\in G} L\cdot u_\sigma, \quad
u_\sigma\cdot u_\tau=f(\sigma,\tau)u_{\sigma\tau},\quad
u_\sigma \cdot \alpha=\sigma(\alpha)\cdot u_\sigma
\]
where $\alpha\in L$, $\sigma,\tau\in G$. The $K$-algebra $(L\circ
G)_f$ is a central simple $K$-algebra.

Suppose that $R$ is a Dedekind domain with quotient field $K$ and $K$ is a number field.
Let $S$ be the integral closure of $R$ in $L$.
Let $h:G \times G\to U(S)$ be a 2-cocycle where $U(S)$ is the group of units in $S$.
Then we may define the crossed-product order, denoted by $(S\circ G)_h$, as follows:
\[
(S\circ G)_h=\bigoplus_{\sigma\in G} S\cdot u_\sigma, \quad
u_\sigma\cdot u_\tau=h(\sigma,\tau)u_{\sigma\tau},\quad
u_\sigma \cdot \alpha=\sigma(\alpha)\cdot u_\sigma
\]
where $\alpha\in S$, $\sigma,\tau\in G$. The $R$-algebra $(S\circ
G)_h$ is an $R$-order in $(L\circ G)_h$.
\end{defn}

%--------------------------t3.2
\begin{theorem}[Williamson, Harada {\cite[page 600, Theorem 28.12; Re, page 375, Theorem 40.15]{CR1}}] \label{t3.2}
Let the notations be the same as in Definition \ref{d3.1}. The
crossed-product $R$-order $(S \circ G)_{h}$ is a hereditary order
if and only if $S/R$ is tamely ramified, i.e.\ for any ramified
prime ideal $Q$ of $S$ over $R$, if $\fn{char} S/Q=p>0$, then
$p\nmid e(Q,L/K)$ where $e(Q,L/K)$ is the ramification index of
$Q$ in $L/K$.
\end{theorem}

%--------------------------d3.3
\begin{defn} \label{d3.3}
The twisted group algebra and the twisted group ring are special cases of the crossed-product algebra and
the crossed-product order when the 2-cocycle $f$, $h$ are the trivial one,
i.e.\ $f(\sigma,\tau)=h(\sigma,\tau)=1$ for all $\sigma,\tau\in G$.

For emphasis, we repeat the definition of a twisted group ring and
denote it by $S\circ G$ (see \cite[page 589]{CR1}). Recall that
$S$ is a Dedekind domain whose quotient field $L$ is an algebraic
number field, $G$ is a finite subgroup of $\fn{Aut}(S)$ with
$R=S^G=\{a\in S:\sigma(a)=a$ for any $\sigma\in G\}$. Then
\[
S\circ G=\bigoplus_{\sigma\in G} S\cdot u_\sigma, \quad
(au_\sigma)\cdot (bu_\tau)=(a\cdot \sigma(b))\cdot u_{\sigma \tau}
\]
where $a,b\in S$, $\sigma,\tau\in G$.
\end{defn}

%------------------------t3.4
\begin{theorem} \label{t3.4}
{\rm (1) (\cite[page 374, Theorem 40.14; CR1, page 591, Theorem
28.5]{Re1})} The twisted group ring $S\circ G$ is a maximal
$R$-order if and only if $S/R$ is unramified.

{\rm (2) (Rosen \cite[page 373, Theorem 40.13]{Ro,Re1})}
The twisted group ring $S\circ G$ is a hereditary $R$-order if and only if $S/R$ is tamely ramified.
\end{theorem}

\begin{remark}
Theorem \ref{t3.4} (2) was obtained first.
Then Theorem \ref{t3.2} came out as the generalization of Theorem \ref{t3.4} (2).
\end{remark}

%------------------------d3.5
\begin{defn} \label{d3.5}
Let $\Lambda=S\circ G$ be a twisted group ring and $L$ be the quotient field of $S$ with $R=S^G$.
We will endow on $S$ a $\Lambda$-module structure by defining $(au_\sigma)\cdot \alpha=a\cdot \sigma(\alpha)$ for any $a,\alpha\in S$, any $\sigma\in G$.
The field $L$ can be given a $\Lambda$-module structure by the same way.
If $J$ is a fractional $S$-ideal of $L$ such that $\sigma(J)\subset J$ for any $\sigma\in G$,
then $J$ becomes a $\Lambda$-submodule of $L$;
such an ideal $J$ is called an ambiguous ideal \cite[page 596; Ro]{CR1}.
\end{defn}

%--------------------------t3.6
\begin{theorem}[Rosen] \label{t3.6}
Let $\Lambda=S\circ G$ be a twisted group ring and $R=S^G$.

{\rm (1) (\cite[Proposition 3; CR1, page 596]{Ro})} Let
$Q_1,Q_2,\ldots,Q_t$ be all the ramified primes of $S$ over $R$
and $e_i=e_i(Q_i,S/R)$ be the ramification index of $Q_i$ for
$1\le i\le t$. For each $1\le i\le t$, let $\{Q_i^{(j)}:1\le j\le
g_i\}$ be the set of $G$-orbits of $Q_i$ (i.e.\
$Q_i^{(j)}=\sigma(Q_i)$ for some $\sigma\in G$, and
$Q_i^{(1)},\ldots,Q_i^{(g_i)}$ are distinct prime ideals of $S$).
Define $J_i=\prod_{1\le j\le g_i} Q_i^{(j)}$. As
$\Lambda$-modules, any ambiguous ideal $J$ is isomorphic to,
$J_1^{a_1}\cdots J_t^{a_t}I$ where $0\le a_i<e_i$ and $I$ is some
ideal of $R$.

{\rm (2) (\cite[Theorem 3.2 and Theorem 3.3]{Ro})}
Assume that $S/R$ is tamely ramified.
Then $\Lambda$ is a left hereditary ring.
The ambiguous ideals $J_1^{a_1} J_2^{a_2} \cdots J_t^{a_t} I$ in $(1)$ are indecomposable projective $\Lambda$-modules.
If $M$ is a $\Lambda$-module such that $M$ is a finitely generated torsion-free $R$-module,
then $M$ is isomorphic to a direct sum of these ambiguous ideals $J_1^{a_1} J_2^{a_2}\cdots J_t^{a_t} I$.
\end{theorem}

\begin{remark}
In the above theorem, it is possible that some of $J_i$ (where $1
\le i \le t$) are isomorphic to each other. See \cite{Ro} for the
uniqueness statement.
\end{remark}

%--------------------------------------S4
\section{A devissage theorem}

If $M$ is a $\bm{Z}\pi$-module, we will write $(M)_0=M/t(M)$ where
$t(M)$ is the torsion submodule of $M$.

First recall Swan's Theorem 5.1 in \cite{Sw6}.

%----------------------t4.1
\begin{theorem} [{\cite[Theorem 5.1 and Corollary 5.2]{Sw6}}] \label{t4.1}
Let $\pi=\langle\sigma\rangle \simeq C_m$ be a cyclic group of
order $m$. Suppose that $n\mid m$ and $M$ is an invertible
$\pi$-lattice. Then
\begin{gather*}
[M/(\sigma^n-1)M]^{fl}=\sum_{d\mid n}  [(M/\Phi_d(\sigma)M)_0]^{fl}, \mbox{ and} \\
[(M/\Phi_n(\sigma)M)_0]^{fl}=\sum_{d\mid n}
\mu\left(\frac{n}{d}\right) [M/(\sigma^d-1)M]^{fl}
\end{gather*}
where $\mu$ is the M\"obius function.
\end{theorem}

\medskip
Examining the proof of the above Theorem \ref{t4.1} in \cite{Sw6}, we find that such a result is valid if $\sigma$ belongs to the center of the group $\pi$. We record this observation as follows.

\begin{theorem}  \label{t4.5}
Let $\pi$ be a finite group and $\sigma \in \pi$ such that $\sigma$ belongs to the center of $\pi$. Suppose that $\langle\sigma\rangle \simeq C_m$ and $n\mid m$. If $M$ is an invertible
$\pi$-lattice, then
\begin{gather*}
[M/(\sigma^n-1)M]^{fl}=\sum_{d\mid n}  [(M/\Phi_d(\sigma)M)_0]^{fl}, \mbox{ and} \\
[(M/\Phi_n(\sigma)M)_0]^{fl}=\sum_{d\mid n}
\mu\left(\frac{n}{d}\right) [M/(\sigma^d-1)M]^{fl}
\end{gather*}
where $\mu$ is the M\"obius function.
\end{theorem}

%---------------------------d4.2
\begin{defn} \label{d4.2}
In Theorem \ref{t4.3} we will consider a finite group $\pi$ which
contains a cyclic normal subgroup $\langle \sigma \rangle$
satisfying that, for any element $\lambda \in \pi$, $\lambda
\sigma \lambda^{-1}= \sigma$ or $\sigma^{-1}$ depending on
$\lambda$; thus we will write $\lambda \sigma \lambda^{-1}=
\sigma^{\epsilon (\lambda)}$ where $\epsilon (\lambda) \in \{1, -1
\}$. Note that $\epsilon (\lambda_1 \lambda_2)=\epsilon
(\lambda_1) \epsilon (\lambda_2)$ for any $\lambda_1, \lambda_2
\in \pi$. Some sample examples of such groups $\pi$ are :
$\pi=\langle
\sigma,\tau:\sigma^m=\tau^2=1,\tau\sigma\tau^{-1}=\sigma^{-1}\rangle\simeq
D_m$, $\pi=\langle\sigma,\rho,\tau:{\sigma}^m=\rho^n=\tau^2=1,\tau\rho=\rho\tau, %
\sigma\rho=\rho\sigma,\tau\sigma\tau^{-1}=\sigma^{-1}\rangle
\simeq C_n \times D_m$, and
$\pi=\langle\sigma,\tau:\sigma^{2m}=\tau^4=1,\sigma^m=\tau^2,\tau\sigma\tau^{-1}=\sigma^{-1}\rangle
\simeq Q_{4m}$,

Let $\pi$ be the group as above and $M$ be a $\pi$-lattice. Define
another $\pi$-lattice $M^*$ as follows. As an abelian group,
$M^*=M$. If $x\in M^*$, $\lambda\in \pi$, denote the scalar
product of $\lambda$ and $x$ in $M^*$ by $\lambda * x$; define
$\lambda
* x:=\epsilon(\lambda) \cdot (\lambda \cdot x)$ where $\lambda\cdot
x$ is the scalar product in $M$.

It is easy to verify that $(M_1\oplus M_2)^*\simeq M_1^*\oplus
M_2^*$ and $(\bm{Z}\pi)^*\simeq \bm{Z}\pi$ (Reason: If
$\sum_{\lambda \in \pi} n_{\lambda}\lambda$ is a general element
in $\bm{Z}\pi$, then a $\bm{Z}\pi$-module isomorphism from
$\bm{Z}\pi$ to $(\bm{Z}\pi)^*$ is given by sending $\sum_{\lambda
\in \pi} n_{\lambda}\lambda$ to $\sum_{\lambda \in \pi}
\epsilon(\lambda) n_{\lambda}\lambda$).

In particular, if $\c{A}$ is a projective ideal over $\bm{Z}\pi$,
$\c{A}^*$ is also a projective ideal over $\bm{Z}\pi$. However, if
$M$ is an invertible $\pi$-lattice, it is not necessary that $M^*$
should be invertible (when $\pi=D_m$, $\bm{Z}^*$ is not invertible).

\end{defn}

%------------------------t4.3
\begin{theorem} \label{t4.3}
Let $\pi$ be a group in Definition \ref{d4.2} with $\langle \sigma
\rangle \simeq C_m$ where $m$ is an integer $\ge 1$ {\rm(}note that
this includes the case of $Q_{4m}$ where $\langle \sigma \rangle
\simeq C_{2m}$, i.e. regarding $2m$ as $m${\rm)}. Let $M$ be an
invertible $\pi$-lattice and $M^*$ be the $\pi$-lattice defined in
Definition \ref{d4.2}. Suppose that $n\mid m$.

{\rm (i)} If $n=1$ or $n=2$ {\rm(}when $m$ is an even integer{\rm)}, then
\begin{gather*}
[M/(\sigma^n-1)M]^{fl}=\sum_{d\mid n}  [(M/\Phi_d(\sigma)M)_0]^{fl}, \mbox{ and} \\
[(M/\Phi_n(\sigma)M)_0]^{fl}=\sum_{d\mid n}
\mu\left(\frac{n}{d}\right) [M/(\sigma^d-1)M]^{fl}
\end{gather*}
where $\mu$ is the M\"obius function.

{\rm (ii)} If $n \ge 3$, then

\begin{gather}
\begin{split}
[M/(\sigma^n-1)M]^{fl}={}&[(M/\Phi_1(\sigma)M)_0]^{fl}+\sum_{d\mid n, d \ge 3} [(M^*/\Phi_d(\sigma)M^*)_0]^{fl} \mbox{ \rm{(}$n$ is odd\rm{)},} \\ % (1)
[M/(\sigma^n-1)M]^{fl}={}&[(M/\Phi_1(\sigma)M)_0]^{fl}+[(M/\Phi_2(\sigma)M)_0]^{fl} \\
& +\sum_{d\mid n, d \ge 3} [(M^*/\Phi_d(\sigma)M^*)_0]^{fl} \mbox{ \rm{(}$n$ is even\rm{)}, and}
\end{split} \\ % (1a)
[(M^*/\Phi_n(\sigma)M^*)_0]^{fl}=\sum_{d\mid n} \mu \left(\frac{n}{d}\right) [M/(\sigma^d-1)M]^{fl} % (2)
\end{gather}

where $\mu$ is the M\"obius function. In particular,
\begin{gather*}
[\bm{Z}\pi/\langle\sigma^n-1\rangle]^{fl}=\sum_{d\mid n}[\bm{Z}\pi/\langle\Phi_d(\sigma)\rangle]^{fl}, \mbox{ and} \\
[\bm{Z}\pi/\langle\Phi_n(\sigma)\rangle]^{fl}=\sum_{d\mid n}
\mu\left(\frac{n}{d}\right)[\bm{Z}\pi/\langle\sigma^d-1\rangle]^{fl}.
\end{gather*}
\end{theorem}

\begin{proof}
For the case $n=1$ or $2$, we may regard the lattices $M/(\sigma^n-1)M$, $(M/\Phi_d(\sigma)M)_0$ (where $d$ divides $n$) as $\pi'$-lattices where $\pi'=\pi/\langle\sigma^2\rangle$. In this situation, $\sigma$ belongs to the center of $\pi'$. Thus we may apply Theorem \ref{t4.5}.

From now on, we will assume that $n \ge 3$.

We will follow the proof of \cite[Theorem 5.1 and Corollary
5.2]{Sw6} with necessary modifications. The Formulae (1) and (2)
are equivalent. It suffices to prove Formula (2). As to the case
$M=\bm{Z}\pi$, recall that $(\bm{Z}\pi)^*\simeq \bm{Z}\pi$; thus
the formulae for $\bm{Z}\pi$ are consequences of Formulae (1) and
(2).

Note that $M/(\sigma^n-1)M$ is torsion-free, because $M$ is
invertible (see \cite[Lemma 5.3]{Sw6}).

\bigskip
Step 1. Let $n\mid m$ and $f(X)\cdot g(X)\mid X^n-1$. Let
$f(\sigma)\cdot \bm{Z}\pi$ be the right ideal of $\bm{Z}\pi$
generated by $f(\sigma)$. For any $\lambda\in \pi$, since
$\lambda\cdot\sigma\cdot\lambda^{-1}=\sigma^i$ for some $i\ge 1$
with $\gcd\{i,m\}=1$, it follows that $\lambda\cdot f(\sigma)\cdot
\lambda^{-1}=f(\sigma^i)$. For any $d$ with $d\mid m$, we have
$\Phi_d(X^i)=\Phi_d(X)\cdot h(X)$ for some $h(X)\in \bm{Z}[X]$
because $\gcd\{i,d\}=1$. As $f(X)$ is a product of cyclotomic
polynomials $\Phi_d(X)$ with $d\mid n$, it follows that
$f(\sigma^i)\in f(\sigma)\cdot \bm{Z}\pi$.

In summary, the right ideal $f(\sigma)\cdot \bm{Z}\pi$ is also a
left ideal. We write it simply as $\langle f(\sigma)\rangle$.
Hence $\bm{Z}\pi/\langle f(\sigma)\rangle$, $\bm{Z}\pi/\langle
f(\sigma)g(\sigma)\rangle$, $\bm{Z}\pi/\langle g(\sigma)\rangle$
may be regarded as two-sided $\bm{Z}\pi$-modules.

In \cite[Lemma 4.2]{Sw6}, the exact sequence
\[
0\to \bm{Z}\pi/\langle f(\sigma)\rangle \to \bm{Z}\pi/\langle
f(\sigma)g(\sigma)\rangle \to \bm{Z}\pi/\langle g(\sigma)\rangle
\to 0
\]
is considered when $\pi$ is a cyclic group. Note that the map
$\bm{Z}\pi/\langle f(\sigma)\rangle \to \bm{Z}\pi/\langle
f(\sigma)g(\sigma)\rangle$ is the multiplication map by
$g(\sigma)$. A similar exact sequence as above when $\pi$ is the
dihedral group is not available because the multiplication map by
$g(\sigma)$ is not a morphism of $\bm{Z}\pi$-modules. We will give
a modification of it in Step 2.

For another modification, in \cite[pages 246--247]{Sw6}, a
sequence of integers $e_k$ is defined and polynomials $g_k(X),
h_k(X)\in \bm{Z}[X]$ are constructed such that the exact sequences
\begin{equation} % (3)
0\to \bm{Z}\pi/\langle g_k(\sigma)\rangle \to \bm{Z}\pi/\langle
h_k(\sigma)\rangle \to \bm{Z}\pi/\langle \sigma^{e_k}-1\rangle \to
0
\end{equation}
satisfy the conditions
\[
h_{2k}(X)=h_{2k+1}(X) \mbox{ ~and~ } g_{2k-1}(X)=g_{2k}(X).
\]

In fact, it is defined that $g_k(X)=f_{k-1}(X)$ if $k$ is even,
$g_k(X)=f_k(X)$ if $k$ is odd, and $h_k(X)=f_k(X)$ if $k$ is even,
$h_k(X)=f_{k-1}(X)$ if $k$ is odd (see the definition of $f_k(X)$
in \cite[page 246]{Sw6}).

We will find a modification of the above construction in Step 3. A
concrete case will be illustrated in Step 4 to exemplify the ideas
of Step 2 and Step 3. The general case will be proved in Step 6.

\bigskip
Step 2. Recall $\langle \sigma\rangle\simeq C_m$ where $m\ge 3$.

For each $d\mid m$ with $d\ge 3$, define
$\widetilde{\Phi}_d=\sigma^{-\varphi(d)/2}\Phi_d(\sigma)\in
\bm{Z}\pi$ where $\varphi(d)$ is the Euler $\varphi$-function.
Define $\widetilde{\Phi}_1=\sigma-1 $ and, if $m$ is an even
integer, define $\widetilde{\Phi}_2=\sigma+1 \in \bm{Z}\pi$.

Note that $\widetilde{\Phi}_d$ belongs to the center of
$\bm{Z}\pi$ if $d\ge 3$.

To simplify the notation, we will write $(d)$ for the two-sided
ideal $\langle
\widetilde{\Phi}_d\rangle=\langle\Phi_d(\sigma)\rangle$; thus
$(d_1)(d_2)\cdots(d_t)$ is just the principal ideal generated by
$\widetilde{\Phi}_{d_1}\cdot \widetilde{\Phi}_{d_2}\cdots
\widetilde{\Phi}_{d_t}$.

\bigskip
Step 3. Let $p_1,p_2,\ldots,p_r$ be the distinct prime divisors of
$n$; if $n$ is an even integer, we choose to write $p_r=2$, the
last prime divisor of $n$. Define a sequence of integers
$d_0,d_1,\ldots,d_{2^r-1}$ as in \cite[page 246]{Sw6}. Define
$d_0=1$, $d_1=p_1$. If $d_0,d_1,\ldots,d_{2^{s-1}-1}$ have been
defined for $s\ge 1$ (using $p_1,\ldots,p_{s-1}$), define
$d_k=p_sd_{2^s-k-1}$ for $2^{s-1}\le k\le 2^s-1$. Thus
$d_{2^{s-1}},\ldots,d_{2^s-1}$ are
$p_sd_{2^{s-1}-1},\ldots,p_sd_1,p_sd_0$ respectively.

Define $e_k=n/d_k$ for $0\le k\le 2^r-1$.

\medskip
Consider the exact sequence of $\pi$-lattices
\begin{equation} % (4)
0\to \langle \sigma^{e_{2^r-1}}-1\rangle/\langle\sigma^n-1\rangle
\to \bm{Z}\pi/\langle\sigma^n-1\rangle \to
\bm{Z}\pi/\langle\sigma^{e_{2^r-1}}-1\rangle \to 0.
\end{equation}

We designate $I^{(2^r-1)}/J^{(2^r-1)}=\langle
\sigma^{e_{2^r-1}}-1\rangle/\langle\sigma^n-1\rangle$ and define
$I^{(2^r-1)}=\langle\sigma^{e_{2^r-1}}-1\rangle$ and
$J^{(2^r-1)}=\langle\sigma^n-1\rangle$. Then we will construct
two-sided ideals $I^{(k)}$ and $J^{(k)}$ inductively where $1\le
k\le 2^r-2$ so that we have the exact sequences
\begin{equation} % (5)
0\to I^{(k)}/J^{(k)} \to \bm{Z}\pi/J^{(k)} \to
\bm{Z}\pi/\langle\sigma^{e_k}-1\rangle \to 0
\end{equation}
satisfying the conditions
\begin{equation} % (6)
J^{(2k)}=J^{(2k-1)}\mbox{ ~and~ } I^{(2k)}/J^{(2k)}\simeq
I^{(2k+1)}/J^{(2k+1)}.
\end{equation}

\medskip
Assume that we have these exact sequences. Suppose $M$ is an
invertible $\pi$-lattice. Tensor it with these exact sequences (4)
and (5) over $\bm{Z}\pi$. We get an exact sequence of
$\bm{Z}\pi$-modules
\[
\fn{Tor}_1^{\bm{Z}\pi}(\bm{Z}\pi/\langle\sigma^{e_k}-1\rangle,M)\to
(I^{(k)}/J^{(k)}) \otimes_{\bm{Z}\pi} M\to M/J^{(k)}M\to
M/(\sigma^{e_k}-1)M\to 0.
\]

Note that $\fn{Tor}_1^{\bm{Z}\pi}(\bm{Z}\pi/\langle\sigma^{e_k}-1\rangle,M)\otimes_{\bm{Z}} \bm{Q}\simeq %
\fn{Tor}_1^{\bm{Q}\pi}(\bm{Q}\pi/\langle\sigma^{e_k}-1\rangle,\bm{Q}\otimes_{\bm{Z}}
M)$. Since $\bm{Q}\pi$ is a semi-simple algebra, it follows that
$\fn{Tor}_1^{\bm{Q}\pi}(\cdot,\cdot)=0$. Thus the kernel of
$\{(I^{(k)}/J^{(k)})\otimes_{\bm{Z}\pi} M \to M/J^{(k)}M\}$ is
torsion. Now we may apply \cite[Lemma 5.4]{Sw6} because
$M/(\sigma^{e_k}-1)M$ is torsion-free.

\medskip
Define $N_k=((I^{(k)}/J^{(k)})\otimes_{\bm{Z}\pi} M)_0$,
$N'_k=(M/J^{(k)}M)_0$. Thus we obtain exact sequences of
$\pi$-lattices
\[
0\to N_k \to N'_k\to M/(\sigma^{e_k}-1)M\to 0
\]
where $1\le k\le 2^r-1$.

Since $M/(\sigma^{e_k}-1)M$ is invertible, apply Lemma \ref{l2.5}.
We find
\[
[N'_k]^{fl}=[N_k]^{fl}+[M/(\sigma^{e_k}-1)M]^{fl}.
\]

The conditions (6) ensure us to conclude that
\begin{equation} % (7)
[N_1]^{fl}=\sum (-1)^k [M/(\sigma^{e_k}-1)M]^{fl}
\end{equation}
where $0\le k\le 2^r-1$. Because $\mu(d_k)=(-1)^k$ (see \cite[page
247]{Sw6}), we get $[N_1]^{fl}=\sum
\mu(d_k)[M/(\sigma^{n/d_k}-1)M]^{fl}$ where $d_k$ runs over the
square-free divisors of $n$. This finishes the proof of Formula
(2).

Note that $N_1$ may be regarded as a module over
$\bm{Z}\pi/\langle\Phi_n(\sigma)\rangle$. More precisely, we will
show that $N_1$ is isomorphic to $(M^*/\Phi_n(\sigma)M^*)_0$ in
Step 5 and Step 6.

\bigskip
Step 4. As an illustration, consider the case $m=n=pqr$ where $p$,
$q$, $r$ are distinct prime numbers; if $m$ is even, we require
that $r=2$.

By Step 3, $d_0,d_1,d_2,\ldots,d_7$ are $1,p,qp,q,rq,rqp,rp,r$
respectively. Thus $e_0,e_1,\ldots,e_7$ are
$pqr,qr,r,pr,p,1,q,pq$.

\medskip
Define $I^{(7)}=\langle\sigma^{pq}-1\rangle=(1)(p)(q)(pq)$,
$J^{(7)}=\langle\sigma^{pqr}-1\rangle=\prod_{d\mid pqr}(d)$
(remember the notation $(d_1)(d_2)\cdots(d_t)$ in Step 2).

By Formula (5), because $e_6=q$, define
$I^{(6)}=\langle\sigma^q-1\rangle=(1)(q)$.

We will find $J^{(6)}$ such that $I^{(6)}/J^{(6)}\simeq
I^{(7)}/J^{(7)}$ (to ensure the validity of Formula (6)). Define
$J^{(6)}=(1)(q)(r)(pr)(qr)(pqr)$. Since $\widetilde{\Phi}_p,
\widetilde{\Phi}_{pq}$ belong to the center of $\bm{Z}\pi$, the
multiplication by $\widetilde{\Phi}_p \widetilde{\Phi}_{pq}$ is
indeed a $\bm{Z}\pi$-isomorphism from $I^{(6)}/J^{(6)}$ to
$I^{(7)}/J^{(7)}$.

We remark that it is not mysterious at all to find $J^{(6)}$. It
is ``equivalent" to find the product of cyclotomic polynomials
$(X^{pqr}-1)(X^{pq}-1)^{-1}(X^q-1)=(X^n-1)(X^{e_7}-1)^{-1}(X^{e_6}-1)$.

\medskip
Now we turn to $I^{(5)}$ and $J^{(5)}$ using the formulae (5) and
(6).

By Formula (6), define $J^{(5)}=J^{(6)}=(1)(q)(r)(pr)(qr)(pqr)$.
By Formula (5), define $I^{(5)}=\langle\sigma^{e_5}-1\rangle=(1)$.
We get the exact sequence $0\to I^{(5)}/J^{(5)}\to
\bm{Z}\pi/J^{(5)}\to \bm{Z}\pi/\langle\sigma^{e_5}-1\rangle\to 0$
automatically.

\medskip
Proceed as before. Define
$I^{(4)}=\langle\sigma^{e_4}-1\rangle=(1)(p)$. We will find
$J^{(4)}$ such that $I^{(4)}/J^{(4)}\simeq I^{(5)}/J^{(5)}$.
Define $J^{(4)}=(1)(p)(q)(r)(pr)(qr)(pqr)$. Then the
multiplication by $\widetilde{\Phi}_p$ gives the isomorphism
$I^{(5)}/J^{(5)}\simeq I^{(4)}/J^{(4)}$.

Define $J^{(3)}=J^{(4)}=(1)(p)(q)(r)(pr)(qr)(pqr)$,
$I^{(3)}=\langle\sigma^{e_3}-1\rangle=(1)(p)(r)(pr)$.

Define $I^{(2)}=\langle\sigma^{e_2}-1\rangle=(1)(r)$,
$J^{(2)}=(1)(q)(r)(qr)(pqr)$. The multiplication by
$\widetilde{\Phi}_p\widetilde{\Phi}_{pr}$ gives the isomorphism
$I^{(2)}/J^{(2)}\simeq I^{(3)}/J^{(3)}$.

Define $J^{(1)}=J^{(2)}=(1)(q)(r)(qr)(pqr)$,
$I^{(1)}=\langle\sigma^{e_1}-1\rangle=(1)(q)(r)(qr)$.

Note that $I^{(1)}/J^{(1)}$ is a module over
$\bm{Z}\pi/\langle\Phi_{pqr}(\sigma)\rangle$. Moreover,
$\widetilde{\Phi}_1$ and $\widetilde{\Phi}_2$ do not appear in the
multiplication maps associated to the isomorphisms
$I^{(7)}/J^{(7)} \simeq I^{(6)}/J^{(6)}$, $I^{(5)}/J^{(5)}\simeq
I^{(4)}/J^{(4)}$ and $I^{(3)}/J^{(3)}\simeq I^{(2)}/J^{(2)}$.

\bigskip
Step 5. Continue the discussion of Step 4 for the case $m=pqr$.
When $M$ is an invertible $\pi$-lattice, we will show that
$(I^{(1)}/J^{(1)})\otimes_{\bm{Z}\pi} M\simeq
M^*/\Phi_{pqr}(\sigma)\cdot M^*$ where $I^{(1)}$, $J^{(1)}$ are
the two-sided ideals as in Step 4.

Recall that $I^{(1)}=\langle \sigma^{e_1} -1\rangle$ where
$e_1=n/p$ (in the general case, $e_1=n/p_1$). We may regard
$I^{(1)}/J^{(1)}$ as a two-sided cyclic $\bm{Z}\pi$-module with a
suitable generator $u$.

Case 1. The order of $\sigma$ is an odd integer $m$.

Define $\Psi_1=\sigma - \sigma^{-1} \in \bm{Z}\pi$. Note that the
two-sided ideals $\Psi_1\cdot \bm{Z}\pi$ and $(\sigma - 1)\cdot
\bm{Z}\pi$ are identical to each other. For, from
$1+\sigma+\cdots+\sigma^{m-1}=1+(1+\sigma)(\sigma+\sigma^3+\cdots+\sigma^{m-2})$
(because $m$ is odd), multiply both sides by $1-\sigma$. We get
$1-\sigma \in \langle 1-\sigma^2 \rangle$.

Also note that, for any $\lambda \in \pi$, $\lambda \Psi_1
\lambda^{-1}=\epsilon(\lambda)\cdot \Psi_1$.

Since $I^{(1)}=\langle \sigma^{e_1} -1 \rangle=(1)(q)(r)(qr)$, we
may define a generator $u$ of $I^{(1)}/J^{(1)}$ by $u:=\Psi_1
\widetilde{\Phi}_q \widetilde{\Phi}_r\widetilde{\Phi}_{qr}\in
I^{(1)}/J^{(1)}$. In $\bm{Z}\pi$, note that $\lambda u
\lambda^{-1}=\epsilon(\lambda)u$.

Define a map $\psi:M^*\to (I^{(1)}/J^{(1)})\otimes_{\bm{Z}\pi} M$
by $\psi(x)=u\otimes x$ where $x\in M^*$. By the definition of
$M^*$, it is easy to verify that $\psi$ is a $\bm{Z}\pi$-morphism.
Thus $\psi$ is a surjective $\bm{Z}\pi$-morphism. When
$M=(\bm{Z}\pi)^{(a)}$ (a free module), it is easy to see that
$\psi$ is a $\bm{Z}\pi$-isomorphism. In general, take a right
exact sequence $(\bm{Z}\pi)^{(a)}\to (\bm{Z}\pi)^{(b)} \to M\to
0$. By the Three-Lemma, we find an isomorphism
$\widetilde{\psi}:M^*/\Phi_{pqr}(\sigma)M^*\simeq
(I^{(1)}/J^{(1)}) \otimes_{\bm{Z}\pi}M$.

\medskip
Case 2. The order of $\sigma$ is an even integer $m$.

Note that $e_1$ is an even integer because we require that $r=2$
(in the general case, $p_r=2$). Thus $I^{(1)}=\langle \sigma^{e_1}
-1 \rangle =\langle\sigma^{e_1/2}-\sigma^{-e_1/2}\rangle$

Now define the generator $u$ of $I^{(1)}/J^{(1)}$ by
$u:=\sigma^{e_1/2}-\sigma^{-e_1/2}$. In $\bm{Z}\pi$, note that
$\lambda u \lambda^{-1}=\epsilon(\lambda) u$ for any $\lambda \in
\pi$.

Define a map $\psi:M^*\to (I^{(1)}/J^{(1)})\otimes_{\bm{Z}\pi} M$
by $\psi(x)=u\otimes x$ where $x\in M^*$. As before, it is not
difficult to show that we have an isomorphism
$\widetilde{\psi}:M^*/\Phi_n(\sigma)M^*\simeq (I^{(1)}/J^{(1)})
\otimes_{\bm{Z}\pi}M$.

\bigskip
Step 6. Now consider the general case where $n$ is any integer
$\ge 3$ with $n\mid m$. We keep the notations in Step 3; in
particular, recall the prime divisors of $n$ and the integers
$d_k, e_k$ when $0 \le k \le 2^r - 1$.

We will define monic polynomials $F_k \in \bm{Z}[X]$ (where $1 \le
k \le 2^r -1$). Using these integral polynomials, define $J^{(k)}
= \langle F_k(\sigma) \rangle$ the principal two-sided ideal
generated by $F_k(\sigma)$. We also define $I^{(k)} = \langle
\sigma^{e_k} -1 \rangle$. Then these $I^{(k)}$ and $J^{(k)}$
satisfy the conditions (4), (5), (6) in Step 3.

\medskip
First of all, define $E_k (X) = X^{e_k} -1 \in \bm{Z}[X]$ where $0
\le k \le 2^r -1$. We will write $E_k$ for $E_k(X)$ in the sequel.

For $1 \le k \le 2^r -1$, define $F_k \in \bm{Z}[X]$ as follows.
Define $F_{2^r-1}=E_0$, $F_1=F_2, F_3=F_4, \dots, F_{2k-1}=F_{2k},
\ldots$ $F_{2^r-3}=F_{2^r-2}$; define
$F_{2k}=E_{2k}E_{2k+1}^{-1}E_{2k+2}E_{2k+3}^{-1} \cdots E_{2^r
-1}^{-1} E_0$ if $2 \le 2k \le 2^r -2$.

For $1 \le k \le 2^r -1$, define $G_k \in \bm{Z}[X]$ which will be
used in the multiplication isomorphisms of $I^{(2k)}/J^{(2k)}$ and
$I^{(2k+1)}/J^{(2k+1)}$. Define $G_2=G_3, G_4=G_5, \dots,
G_{2k}=G_{2k+1}, \ldots$, $G_{2^r-2}=G_{2^r-1}$; define
$G_{2k+1}=E_{2k+1}^{-1}E_{2k+2}E_{2k+3}^{-1}E_{2k+4} \ldots E_{2^r
-1}^{-1} E_0$ if $1 \le 2k+1 \le 2^r -1$.

It is clear that $F_k=E_k G_k$ for $1 \le k \le 2^r-1$.

Furthermore, all of these $F_k$ and $G_k$ are monic polynomials in
$\bm{Z}[X]$ and each of them divides $X^n -1$. Just compare the
definitions of $F_k$ and $G_k$ with those of $f_k(X)$ in
\cite[Lemma 5.6]{Sw6}: $f_k(X)=E_0E_1^{-1}E_2E_3^{-1} \cdots
E_k^{(-1)^k}$. It is known that $f_k(X)$ is a monic polynomial
dividing $X^n -1$ (see \cite[Lemma 5.6]{Sw6}). It is not difficult
to apply the same method in the proof of Lemma 5.6 of \cite{Sw6}
to prove the same results for $F_k$ and $G_k$.

\medskip
Define $I^{(k)} =\langle E_k(\sigma) \rangle= \langle \sigma^{e_k}
-1 \rangle$ and $J^{(k)} = \langle F_k(\sigma) \rangle$ for $1 \le
k \le 2^{r}-1$.

When $0 \le k \le 2^{r}-1$, it is routine to verify that
$d_{2k+1}/d_{2k}=p_1^{(-1)^k}$. Note that
$G_{2k}(\sigma)=G_{2k+1}(\sigma)$. Now it is easy to show that (i)
if $k$ is an even integer, the isomorphism of
$I^{(2k+1)}/J^{(2k+1)}$ to $I^{(2k)}/J^{(2k)}$ is given by the
multiplication by $\prod_{t \mid e_{2k+1}}\tilde{\Phi}_{tp_1}$
(because $(X^{e_{2k+1}}-1)\prod_{t\mid
e_{2k+1}}\Phi_{tp_1}(X)=X^{e_{2k}}-1$), and (ii) if $k$ is an odd
integer, the isomorphism of $I^{(2k)}/J^{(2k)}$ to
$I^{(2k+1)}/J^{(2k+1)}$ is given by the multiplication by
$\prod_{t \mid e_{2k}}\tilde{\Phi}_{tp_1}$ (because
$(X^{e_{2k}}-1)\prod_{t\mid
e_{2k}}\Phi_{tp_1}(X)=X^{e_{2k+1}}-1$). Hence the isomorphism of
Formula (6) in Step 3 is proved.

Note that $G_1=\langle\Phi_n(X)\rangle$. If we write
$I^{(1)}=\prod_{1\le i\le t}(d_i)$, then
$J^{(1)}=(n)\cdot\prod_{1\le i\le t} (d_i)$. It follows that
$I^{(1)}/J^{(1)}$ is a module over $\bm{Z}\pi/\langle
\Phi_n(\sigma)\rangle$. Now we may show that
$I^{(1)}/J^{(1)}\otimes_{\bm{Z}\pi} M\simeq M^*/\Phi_n(\sigma)M^*$
by the same method as in Step 5.

\bigskip

\end{proof}

\begin{lemma} \label{l4.6}
Let $\zeta$ be a primitive $m$-th root of unity in $\bm{C}$ where $m$ is an odd integer $\ge 3$. Let $\pi= \langle \tau \rangle \simeq C_2$ act on $\bm{Z}[\zeta]$ by $\tau \cdot \zeta=\zeta^{-1}$. Then $\bm{Z}[\zeta] \simeq (\bm{Z}\pi)^{(\phi(m)/2)}$ as $\pi$-lattices.
\end{lemma}

\begin{proof}
Step 1. Denote by $\Phi_d(X)$ the $d$-th cyclotomic polynomial. Define $\Psi_d(X)=(X^d -1)/(X-1)$.

Write $m=\prod_{1 \le i \le t}p_i^{a_i}$ where $p_1, \ldots , p_t$ are distinct prime numbers and $a_i \ge 1$. We will show that the coefficient of $X^{\phi(m)/2}$ in $\Phi_m(X)$ is an odd integer. This is obvious if $m$ is an odd prime power. From now on we may assume that $t\ge 2$ in the above prime decomposition of $m$.

Define a monic polynomial $F(X) \in \bm{Z}[X]$ by $\Psi_m(X)=F(X) \prod_{1 \le i \le t}\Psi_{p^{a_i}}(X)$. It follows that $F(1)=1$. Since $\Phi_m(X)$ is a factor of $F(X)$, thus $\Phi_m(1)=1$ or $-1$.

From $\Phi_m(X)=X^{\phi(m)}\Phi_m(X^{-1})$, we find that $\Phi_m(X)$ is ``symmetric" with respect to the term $X^{\phi(m)/2}$. Because $\Phi_m(1)=1$ or $-1$, thus the coefficient of $X^{\phi(m)/2}$ in $\Phi_m(X)$ is an odd integer.

\bigskip
Step 2. Let $c$ be the coefficient of $X^{\phi(m)/2}$ in $\Phi_m(X)$. It follows that $c\zeta^{\phi(m)/2} \in \sum_{0 \le j \le (\phi(m)/2) -1}\bm{Z} \cdot \zeta^j + \sum_{(\phi(m)/2) +1 \le j \le \phi(m)}\bm{Z} \cdot \zeta^j$.

Let $\bm{Z}_{(2)}$ be the localization of $\bm{Z}$ at the prime ideal $(2)=2\bm{Z}$. Hence we have $c\zeta^{\phi(m)/2} \in \sum_{0 \le j \le (\phi(m)/2) -1}\bm{Z}_{(2)} \cdot \zeta^j + \sum_{(\phi(m)/2) +1 \le j \le \phi(m)}\bm{Z}_{(2)} \cdot \zeta^j$. It follows that $1 \in \sum_{1 \le j \le \phi(m)/2}(\bm{Z}_{(2)} \cdot \zeta^{-j}+\bm{Z}_{(2)} \cdot \zeta^j)$ and $\bm{Z}_{(2)}[\zeta]= \oplus_{1 \le j \le \phi(m)/2}(\bm{Z}_{(2)} \cdot \zeta^{-j}+\bm{Z}_{(2)} \cdot \zeta^j)$.

\bigskip
Step 3. The action of $\pi$ induces an action on $\bm{Z}_{(2)}[\zeta]$. From Step 2, we find that $\bm{Z}_{(2)}[\zeta] \simeq (\bm{Z}_{(2)}\pi)^{(\phi(m)/2)}$, a free module over $\bm{Z}_{(2)}\pi$. Thus $Ext_{\bm{Z}_{(2)}\pi}^1(\bm{Z}_{(2)}[\zeta], \cdot)$ is a zero functor.

By Reiner's Theorem \cite[page 74, Theorem 4.19]{Sw3}, $\bm{Z}[\zeta]$ is a direct sum of $\bm{Z}$, $\bm{Z}_{-}$, or $\bm{Z}\pi$. If $\bm{Z}$ or $\bm{Z}_{-}$ is a direct summand of $\bm{Z}[\zeta]$, then $\bm{Z}_{(2)}$ or $(\bm{Z}_{(2)})_{-}$ is also a direct summand of $\bm{Z}_{(2)}[\zeta]$. By \cite{Sw1}, neither $\bm{Z}_{(2)}$ nor $(\bm{Z}_{(2)})_{-}$ is a projective module over $\bm{Z}_{(2)}\pi$ (by counting the ranks). Hence $Ext_{\bm{Z}_{(2)}\pi}^1(\bm{Z}_{(2)}[\zeta], \cdot)$ would not be a zero functor. This leads to a contradiction. Thus the only direct summands of $\bm{Z}[\zeta]$ are $\bm{Z}\pi$'s, i.e. $\bm{Z}[\zeta]$ is a free module.
\end{proof}

The following theorem is a variant of the above Theorem \ref{t4.3}
for the group $C_{q^f}\times D_m$.

\medskip
\begin{theorem} \label{t4.4}
Let $\pi=\langle\sigma,\rho,\tau:{\rho}^{q^f}=\sigma^m=\tau^2=1$,
$\tau\rho\tau^{-1}=\rho$, $\tau\sigma\tau^{-1}=\sigma^{-1}$,
$\sigma\rho=\rho\sigma\rangle\simeq C_{q^f}\times D_m$ where $q$
is an odd prime number, $f \ge 1$, $m$ is an odd integer with
$m\ge 3$ and gcd $\{q,m \}=1$. Suppose that $M$ is an invertible
$\pi$-lattice and $M^*$ is the $\pi$-lattice defined in Definition
\ref{d4.2}. Then
\begin{gather*}
\begin{split}
[M]^{fl}={}& [(M/\Phi_1(\sigma)M)_0]^{fl}+\sum_{d\mid m, d \ge 3} [(M^*/\Phi_d(\sigma)M^*)_0]^{fl} \mbox{ \rm{(}$m$ is odd\rm{)}}, \\
[M]^{fl}={}& [(M/\Phi_1(\sigma)M)_0]^{fl}+[(M/\Phi_2(\sigma)M)_0]^{fl} \\
 & +\sum_{d\mid m, d \ge 3} [(M^*/\Phi_d(\sigma)M^*)_0]^{fl} \mbox{ \rm{(}$m$ is even\rm{)}},
\end{split}   \\
[(M^*/\Phi_m(\sigma)M^*)_0]^{fl}=\sum_{d\mid
m}\mu\left(\frac{m}{d}\right)[M/(\sigma^d-1)M]^{fl}, \\
[(M^*/\Phi_m(\sigma)M^*)_0]^{fl}=[(M^*/\langle\Phi_{q^f}(\rho),
\Phi_m(\sigma)\rangle M^*)_0]^{fl}+[M^{\prime}]^{fl}
\end{gather*}
where $M^{\prime}$ is a lattice over $\bm{Z}\pi/\pi^{\prime}$ for
some normal subgroup $\pi^{\prime}$ with $\pi^{\prime}\neq \{1
\}$.
\end{theorem}

\begin{proof}
The first two formulae follow from Theorem \ref{t4.3}. It remains
to prove the last formula.

\bigskip
Step 1. Consider the exact sequence $0 \to \bm{Z}\pi/\langle
\Phi_{q^f}(\rho), \Phi_m(\sigma)\rangle \to \bm{Z}\pi/\langle
\Phi_m(\sigma)\rangle \to \bm{Z}\pi/\langle \rho^{q^{f-1}}-1,
\Phi_m(\sigma)\rangle \to 0$ where the first map is defined to be
the multiplication map by $\rho^{q^{f-1}}-1$.

Tensor it with $M^*$ over $\bm{Z}\pi$. Because $M$ is an
invertible lattice, it is not difficult to verify that the
following is an exact sequence of $\pi$-lattices
\[
0 \to (M^*/\langle \Phi_{q^f}(\rho), \Phi_m(\sigma)\rangle M^*)_0
\to (M^*/\Phi_m(\sigma) M^*)_0 \to (M^*/\langle \rho^{q^{f-1}}-1,
\Phi_m(\sigma)\rangle M^*)_0 \to 0
\]
We remark that, for any $\pi$-lattice $E$, the lattice
$(E^*/\langle \Phi_{q^f}(\rho), \Phi_m(\sigma)\rangle E^*)_0$ is
isomorphic to the sublattice $(\rho^{q^{f-1}}-1)g(\sigma)\cdot E$
of $E$ where $g(X) \in \bm{Z}[X]$ is defined by $X^m
-1=\Phi_m(X)g(X)$.

\medskip
Write $M^{\prime}:=(M^*/\langle \rho^{q^{f-1}}-1,
\Phi_m(\sigma)\rangle M^*)_0$. $M^{\prime}$ may be regarded as a
lattice over $\bm{Z}\pi/\langle\rho^{q^{f-1}}\rangle$.

In the next step we will show that
$\hat{H}^0(\pi^{\prime},(M^*/\langle \Phi_{q^f}(\rho),
\Phi_m(\sigma)\rangle M^*)_0) =0$ for all subgroups $\pi^{\prime}$
of $\pi$. Assume this result. We may apply Part (2) of Lemma
\ref{l2.5} and obtain the desired formula
\[
[(M^*/\Phi_m(\sigma) M^*)_0]^{fl}=[(M^*/\langle \Phi_{q^f}(\rho),
\Phi_m(\sigma)\rangle M^*)_0]^{fl}+[M^{\prime}]^{fl}.
\]

\bigskip
Step 2. We will show that, for any subgroup $\pi^{\prime}$ of
$\pi$, $\hat{H}^0(\pi^{\prime},N) =0$ where $N:=(M^*/\langle
\Phi_{q^f}(\rho), \Phi_m(\sigma)\rangle M^*)_0$. Note that
$\Phi_{q^f}(\rho)\cdot N=\Phi_m(\sigma)\cdot N=0$. We may also
assume that $\pi^{\prime} \neq \{ 1 \}$ and $N \neq 0$.

Since $M$ is an invertible lattice, there is a $\pi$-lattice
$M^{\prime\prime}$ such that $M \oplus M^{\prime\prime} \simeq
\oplus_{1 \le i \le s}\bm{Z}\pi/\pi_i$ where $\pi_i$'s are some
subgroups of $\pi$. If $\hat{H}^0(\pi^{\prime}, (\bm{Z}\pi/\langle
\Phi_{q^f}(\rho), \Phi_m(\sigma)\rangle
\otimes_{\bm{Z}\pi}(\bm{Z}\pi/\pi_i)^*)_0) =0$ for all $i$, then
$\hat{H}^0(\pi^{\prime},N) =0$. In other words, the problem is
reduced to the case $M=\bm{Z}\pi/\pi_0$ where $\pi_0$ is a
subgroup of $\pi$.

\medskip
We claim that either $\pi_0=\{ 1 \}$ or $\pi_0=\langle \sigma^i
\tau \rangle$ for some integer $i$.

In fact, if $\pi_0 \cap \langle \rho \rangle  \neq \{ 1 \}$, then
$\rho^{q^t} \in \pi_0$ for some $0 \le t \le f-1$. Thus we have
$(\rho^{q^t}- 1) \cdot N=0$. But we also have
$\Phi_{q^f}(\rho)\cdot N=0$. Since $N$ is torsion-free, this leads
to the conclusion that $N=0$, which has been ruled out at the
beginning of this step. Similarly, it can be shown that $\pi_0
\cap \langle \sigma \rangle = \{ 1 \}$.

In summary, we may assume that $M=\bm{Z}\pi$ or
$M=\bm{Z}\pi/\langle \sigma^i \tau \rangle$. The situation of
$M=\bm{Z}\pi/\langle \sigma^i \tau \rangle$ is similar to that of
$M=\bm{Z}\pi/\langle \tau \rangle$. Thus we consider only the
cases $M=\bm{Z}\pi$ and $M=\bm{Z}\pi/\langle \tau \rangle$.

Denote by $\xi_1$ and $\xi_2$ the images of $\rho$ and $\sigma$ in
$N$ respectively. Note that $\xi_1$ (resp. $\xi_2$) is a primitive
$q^f$-th root of unity (resp. a primitive $m$-th root of unity).
We will write $N=\bm{Z}[\xi_1,\xi_2,\tau]$ or
$N=\bm{Z}[\xi_1,\xi_2]$ in the sequel.

\medskip
For each prime divisor $p$ of $\mid\pi^{\prime}\mid$, choose a
$p$-Sylow subgroup $\pi_p^{\prime}$ of $\pi^{\prime}$. Since
$\hat{H}^0(\pi^{\prime}, N) \to \oplus_p \hat{H}^0(\pi^{\prime}_p,
N)$ is injective, it suffices to show that
$\hat{H}^0(\pi^{\prime}_p, N)=0$, i.e. without loss of generality,
we may assume that $\pi^{\prime}$ is a $p$-group.

\medskip
Subcase 1. $p=q$ and $\pi^{\prime}=\langle \rho^{\prime} \rangle
\simeq C_{q^s}$.

Regard $N$ as a $\pi^{\prime}$-lattice. Then $N \simeq
\bm{Z}[\xi_1]^{(a)}$ for some integer $a$. Let $\zeta_{q^s} \in
\langle \xi_1 \rangle$ be a primitive $q^s$-th root of unity.
Regard $\bm{Z}[\xi_1]$ as a module over $\bm{Z}[\zeta_{q^s}]$; it
is isomorphic to a direct sum of a free module and an ideal. Thus
$N \simeq \bm{Z}[\xi_1]^{(a)} \simeq \bm{Z}[\zeta_{q^s}]^{(b)}
\oplus I$ where $b$ is some integer and $I$ is a non-zero ideal of
$\bm{Z}[\zeta_{q^s}]$. We will show that
$\hat{H}^0(\pi^{\prime},\bm{Z}[\zeta_{q^s}])
=0=\hat{H}^0(\pi^{\prime},I)$.

Extend the action of $\rho^{\prime}$ to $\bm{Q}(\zeta_{q^s})$ and
$\bm{Q}I(=\bm{Q}(\zeta_{q^s}))$. The characteristic polynomial of
this linear map $\rho^{\prime}$ is $\Phi_{q^s}(X)$. Since
$\Phi_{q^s}(1) \neq 0$, there is no vector in
$\bm{Q}(\zeta_{q^s})$, which is left fixed by $\rho^{\prime}$.
Hence $\hat{H}^0(\pi^{\prime},\bm{Z}[\zeta_{q^s}]) =0$ and
$\hat{H}^0(\pi^{\prime},I)=0$.

\medskip
Subcase 2. $p$ is a divisor of $m$ and $\mid\pi^{\prime}\mid=p^s$.

Regard $N$ as a $\pi^{\prime}$-lattice. Then $N \simeq
\bm{Z}[\xi_2]^{(c)}$ for some integer $c$. Let $\zeta_{p^s} \in
\langle \xi_2 \rangle$ be a primitive $p^s$-th root of unity. The
remaining proof is similar to that of Subcase 1.

\medskip
Subcase 3. $p=2$ and $\pi^{\prime}= \langle \sigma^i\tau \rangle$.

Because $m$ is odd, $\sigma^i\tau$ is conjugate to $\tau$. Thus
$\hat{H}^0(\langle \sigma^i\tau \rangle, N) \simeq
\hat{H}^0(\langle \tau \rangle, N)$. Hence we may assume that
$\pi^{\prime}= \langle \tau \rangle$ and it remains to show that
$\hat{H}^0(\langle \tau \rangle, N)=0$.

Suppose that $M=\bm{Z}\pi/\langle \tau \rangle$ (the situation for $M=\bm{Z}\pi$ is similar).

Write $N=\oplus_{i,j}\bm{Z}\xi_1^i\xi_2^j$ where $\tau \cdot
\xi_1^i\xi_2^j=-\xi_1^i\xi_2^{-j}$ and $0 \le i \le \phi(q^f)-1, 0 \le j \le \phi(m)-1$.
We will show that $N \simeq
(\bm{Z}\pi^{\prime})^{(d)}$ for some integer $d$. Assume this. From
$\hat{H}^0(\langle \tau \rangle,\bm{Z}\pi^{\prime})=0$, it follows
that $\hat{H}^0(\langle \tau \rangle, N)=0$.

The proof that $N \simeq
(\bm{Z}\pi^{\prime})^{(d)}$ is similar to that of Lemma \ref{l4.6}. In fact, let $\bm{Z}_{(2)}$ be the localization of $\bm{Z}$ at the prime ideal $(2)=2\bm{Z}$. By the same arguments as in Step 1 and Step 2 of the proof of Lemma \ref{l4.6}, we find that $\bm{Z}_{(2)}\otimes_{\bm{Z}} N =\oplus_{0 \le i \le \phi(q^f)-1}\oplus_{1 \le j \le \phi(m)/2} (\bm{Z}_{(2)}\xi_1^i\xi_2^{-j}+\bm{Z}_{(2)}\xi_1^i\xi_2^j)$. Hence $\bm{Z}_{(2)}\otimes_{\bm{Z}} N\simeq
(\bm{Z}_{(2)}\pi^{\prime})^{(d)}$ for some integer $d$. Then apply Reiner's Theorem to $N$ as in Step 3 of the proof of Lemma \ref{l4.6}. We find that $N \simeq
(\bm{Z}\pi^{\prime})^{(d)}$.
\end{proof}

\bigskip
\begin{remark}
R. G.\ Swan kindly pointed out an error in the first version of
this article. The exact sequence $0\to \bm{Z}\pi/\langle
f(\sigma)\rangle\to \bm{Z}\pi/\langle f(\sigma)g(\sigma)\rangle\to
\bm{Z}\pi/\langle g(\sigma)\rangle \to 0$ in \cite[pages
246-247]{Sw6} should be taken carefully; it is not available if
$\pi$ is a non-abelian group.

We remark also that the five short exact sequences in the middle
of page 96 of \cite{EM4} were meant to be the exact sequences
$0\to \bm{Z}\pi/\langle g_k(\sigma)\rangle \to \bm{Z}\pi/\langle
h_k(\sigma)\rangle \to \bm{Z}\pi/\langle
\sigma^{e_k}-1\rangle\allowbreak \to 0$ of Formula (3) in Step 1
of the proof of Theorem \ref{t4.3}. However, as pointed out
before, when $\pi$ is a non-abelian group (e.g.\ $\pi=D_m$,
$C_n\times D_m$, $Q_{4m}$), there seems no obvious reason why the
maps in these exact sequences should be $\bm{Z}\pi$-morphisms.

\end{remark}

%----------------------------------------S5
\section{Proof of Theorem \ref{t1.4}}

We will devote this section to proving ``$(1)\Rightarrow (2)$ of
Theorem \ref{t1.4}". The goal is to show that $T(\pi)=T^g(\pi)$,
$T^g(\pi)\simeq C(\bm{Z}\pi)/C^q(\bm{Z}\pi)$ and
$C(\bm{Z}\pi)/C^q(\bm{Z}\pi) \simeq C(\Omega_{\bm{Z}\pi})$.

The key idea for the proof of $T^g(\pi)\simeq
C(\bm{Z}\pi)/C^q(\bm{Z}\pi)$ is as follows. We use Theorem
\ref{t2.14} (1) to prove $T^g(\pi)\simeq
C(\bm{Z}\pi)/C^q(\bm{Z}\pi)$, i.e.\ for any invertible
$\pi$-lattice $M$, we will find a projective ideal $\c{A}$ over
$\bm{Z}\pi$ such that $[M]^{fl}=[\c{A}]^{fl}$. For this purpose,
we apply Theorem \ref{t4.1}, Theorem \ref{t4.3}, etc. and reduce
the question to the situation of $(M/\Phi_d(\sigma)M)_0$ or
$(M^*/\Phi_d(\sigma)M^*)_0$ for any $d\mid n$ where
$\langle\sigma\rangle\simeq C_n$ (the case $\pi=C_n \times D_m$
requires some more efforts). Since $(M/\Phi_d(\sigma)M)_0$ or
$(M^*/\Phi_d(\sigma)M^*)_0$ is a torsion-free module over
$\bm{Z}\pi/\langle\Phi_d(\sigma)\rangle$, it is important to
understand the structure of torsion-free modules over the
$\bm{Z}$-order
$\Lambda_d:=\bm{Z}\pi/\langle\Phi_d(\sigma)\rangle$. In most
situations, $\Lambda_d$ is a Dedekind domain, a twisted group ring
or a maximal $\bm{Z}$-order. Thus the results in Section 3 are
applicable. The final blow is to use Theorem \ref{t2.10}, i.e.\
Jacobinski-Roiter's Theorem, to find the projective ideal
$\c{A}_d$ for $(M/\Phi_d(\sigma)M)_0$ or
$(M^*/\Phi_d(\sigma)M^*)_0$.

\bigskip
Before the proof, we recall some basic facts of maximal orders
\cite[Section 26; Re]{CR1}.

%----------------------d5.1
\begin{defn} \label{d5.1}
Let $K$ be a field, $\Sigma$ be a finite-dimensional separable
algebra over $K$. Let $R$ be a Dedekind domain with quotient field
$K$, and $\Lambda$ be a maximal $R$-order in $\Sigma$. A finitely
generated left $\Lambda$-module $M$ is called a $\Lambda$-lattice
if it is a projective $R$-module.
\end{defn}

%-----------------------t5.2
\begin{theorem}[{\cite[page 565]{CR1}}] \label{t5.2}
Let the notations be the same as in Definition \ref{d5.1}, let
$\Lambda$ be a maximal $R$-order and $K$ be the quotient field of
$R$. Then

{\rm (1)} $\Lambda$ is a left and right hereditary ring.

{\rm (2)} Every left $\Lambda$-lattice is $\Lambda$-projective.

{\rm (3)} A left $\Lambda$-lattice $M$ is indecomposable if and
only if $KM$ is a simple module over $K\Lambda$.
\end{theorem}

%----------------------t5.4
\begin{theorem}[{\cite[page 176, Theorem 18.4]{Re1}}] \label{t5.4}
Let $R$ be a DVR with quotient field $K$, $\Lambda$ be an
$R$-order such that $K\Lambda$ is a central simple algebra over a
field containing $K$. Then $\Lambda$ is a maximal order if and
only if $\Lambda$ is a hereditary ring and $\fn{rad}(\Lambda)$ is
a maximal two-sided ideal where $\fn{rad}(\Lambda)$ is the
Jacobson radical of $\Lambda$.
\end{theorem}

\bigskip
\begin{proof}[Proof of $(1)\Rightarrow (2)$ of Theorem \ref{t1.4}] ~ \par
First we show that $T(\pi)=T^g(\pi)$.

Note that $\pi$ is a group such that all Sylow subgroups of $\pi$
are cyclic. Hence any coflabby $\pi$-lattice is invertible by
Theorem \ref{t2.7}. On the other hand, if $M$ is any
$\pi$-lattice, by Lemma \ref{l2.6} (2), there is an exact sequence
$0\to M\to C\to P\to 0$ where $C$ is coflabby (and hence
invertible) and $P$ is permutation. Apply Lemma \ref{l2.5} (1). We
get $[C]^{fl}=[M]^{fl}+[P]^{fl}=[M]^{fl}$.

In conclusion, for any $\pi$-lattice $M$ with $[M]\in T(\pi)$,
there is an invertible $\pi$-lattice $C$ such that $[M]=[C]$ in
$T(\pi)$ by Lemma \ref{l2.2}. Thus $T(\pi)=T^g(\pi)$.

\bigskip
We will use \cite[page 707, Theorem 4.2]{EM2} to show that
$C(\bm{Z}\pi)/C^q(\bm{Z}\pi)\simeq C(\Omega_{\bm{Z}\pi})$, i.e.\
$C^q(\bm{Z}\pi)=\widetilde{C}^q(\bm{Z}\pi)=\widetilde{C}(\bm{Z}\pi)$
(see Definition \ref{d2.13}).

We remark first that \cite[page 707, line 6]{EM2} contains a
misprint which should be corrected as follows: If $\pi=C\rtimes P$
where $C=\langle\sigma\rangle\simeq C_n$ is normal in $\pi$ and
$m\mid n$, the natural map $\mu_m:P\to
\fn{Aut}(C/\langle\sigma^m\rangle)$ is induced from the action of
$P$ on $C$ and $C/\langle \sigma^m\rangle$. Define
$D_m=\fn{Ker}(\mu_m)$. In particular, $P_n=\{\lambda\in P:\lambda
\sigma\lambda^{-1}=\sigma\}$.

Return to the group $\pi$ in (1) of Theorem \ref{t1.4}. We may
write $\pi=C\rtimes P$ where $C$ is cyclic, $P=\{1\}$, $C_2$ or
$C_4$. Thus the assumptions (c) or (d) of \cite[Theorem 4.2]{EM2}
are fulfilled. We conclude that
$C^q(\bm{Z}\pi)=\widetilde{C}^q(\bm{Z}\pi)=\widetilde{C}(\bm{Z}\pi)$.
Hence $C(\bm{Z}\pi)/C^q(\bm{Z}\pi)\simeq C(\Omega_{\bm{Z}\pi})$.

\bigskip
By Theorem \ref{t2.14}, to show that $T^g(\pi)\simeq
C(\bm{Z}\pi)/C^q(\bm{Z}\pi)$, it suffices to show that, for any
invertible $\pi$-lattice $M$, there is a projective ideal $\c{A}$
over $\bm{Z}\pi$ such that $[M]^{fl}=[\c{A}]^{fl}$. This is what
we will prove in the sequel. The ideas of the proof have been
explained at the beginning of this section. Once it is proved, the
proof of $(1)\Rightarrow (2)$ is finished.

\bigskip
\begin{Case}{1} $\pi=\langle\sigma\rangle\simeq C_n$. \end{Case}

W remark that this case is proved also in \cite[Theorem 2.10]{Sw6}
using some ideas in \cite{EM4}. In the following, we give a
slightly different proof.

\medskip
Step 1. For any $d\mid n$, note that
$\bm{Z}\pi/\langle\Phi_d(\sigma)\rangle \simeq \bm{Z}[\zeta_d]$
where $\zeta_d$ is some primitive $d$-th root of unity.

Moreover, by Theorem \ref{t4.1}, we find that
\[
[\bm{Z}\pi/\langle\Phi_d(\sigma)\rangle]^{fl}=\sum_{e\mid d} \mu\left(\frac{d}{e}\right)[\bm{Z}\pi/\langle \sigma^e-1\rangle]^{fl}.
\]

Since $\bm{Z}\pi/\langle\sigma^e-1\rangle\simeq \bm{Z}\pi/\pi'$
where $\pi'=\langle\sigma^e\rangle$, it follows that
$\bm{Z}\pi/\langle\sigma^e-1\rangle$ is a permutation
$\pi$-lattice and therefore
$[\bm{Z}\pi/\langle\sigma^e-1\rangle]^{fl}=0$. Thus
$[\bm{Z}\pi/\langle\Phi_d(\sigma)\rangle]^{fl}=0$.

\medskip
Step 2. Suppose that $M$ is an invertible $\pi$-lattice.

Apply Theorem \ref{t4.1}. We get
\[
[M]^{fl}=[M/(\sigma^n-1)M]^{fl}=\sum_{d\mid n} [(M/\Phi_d(\sigma)M)_0]^{fl}.
\]

Since $(M/\Phi_d(\sigma) M)_0$ is a torsion-free module over
$\bm{Z}\pi/\langle\Phi_d(\sigma)\rangle\simeq \bm{Z}[\zeta_d]$, it
follows that $(M/\Phi_d(\sigma)M)_0\simeq F_d\oplus J_d$ where
$F_d$ is a free module over
$\bm{Z}\pi/\langle\Phi_d(\sigma)\rangle$, $J_d$ is a non-zero
ideal of $\bm{Z}\pi/\langle \Phi_d(\sigma)\rangle$. Note that
$[F_d]^{fl}=0$ because
$[\bm{Z}\pi/\langle\Phi_d(\sigma)\rangle]^{fl}=0$ by Step 1.

Thus we get $[M]^{fl}=\sum_{d\mid n} [J_d]^{fl}$.

\medskip
Step 3. Each $J_d$ is locally isomorphic to
$\bm{Z}\pi/\langle\Phi_d(\sigma)\rangle$. Hence $J_d$ and
$\bm{Z}\pi/\langle \Phi_d(\sigma)\rangle$ belong to the same
genus. Apply Theorem \ref{t2.10}. We find a projective ideal
$\c{A}_d$ over $\bm{Z}\pi$ such that $J_d\oplus \bm{Z}\pi \simeq
\bm{Z}\pi/\langle\Phi_d(\sigma)\rangle \oplus\c{A}_d$. Hence
$[J_d]^{fl}=[\c{A}_d]^{fl}$ because
$[\bm{Z}\pi/\langle\Phi_d(\sigma)\rangle]^{fl}=0$ by Step 1.

\medskip
Step 4. By Swan's Theorem \cite[page 677, Corollary 32.12]{CR1},
$\bigoplus_{d\mid n} \c{A}_d\simeq F\oplus \c{A}$ for some
projective ideal $\c{A}$ over $\bm{Z}\pi$ and some free
$\bm{Z}\pi$-module $F$. Hence $[M]^{fl}=\sum_{d\mid n}
[J_d]^{fl}=\sum_{d\mid n}[\c{A}_d]^{fl}=[F\oplus
\c{A}]^{fl}=[\c{A}]^{fl}$ as expected.

\medskip
Step 5. We rephrase the above result as a form which will
be used in the sequel.

Let $\pi \simeq C_n$, $M$ be a $\pi$-lattice. Then there is a
projective ideal $\c{A}$ over $\bm{Z}\pi$ such that
$[M]^{fl}=[\c{A}]^{fl}$.

As before, by Lemma \ref{l2.6} (2), find an exact sequence $0\to
M\to N\to P\to 0$ where $N$ is coflabby and $P$ is permutation.
Apply Lemma \ref{l2.5} (1). We get
$[N]^{fl}=[M]^{fl}+[P]^{fl}=[M]^{fl}$.

By Theorem \ref{t2.7} any coflabby $\pi$-lattice is invertible;
thus $N$ is invertible. Since $N$ is invertible, it is possible to
find a projective ideal $\c{A}$ over $\bm{Z}\pi$ such that
$[N]^{fl}=[\c{A}]^{fl}$ from the above Steps 1,2,3 and 4. It
follows that $[M]^{fl}=[\c{A}]^{fl}$.

\bigskip
\begin{Case}{2}
$\pi=\langle\sigma,\tau:\sigma^m=\tau^2=1,\tau^{-1}\sigma\tau=\sigma^{-1}\rangle\simeq D_m$ where $m$ is an odd integer with $m\ge 3$.
\end{Case}

\medskip
\textit{Subcase 2.1} $m=p^c$ where $p$ is an odd prime number and $c\ge 1$.

Step 1. We use similar methods as in Case 1 and apply Theorem
\ref{t4.3}. However, if $d\mid p^c$,
$\bm{Z}\pi/\langle\Phi_d(\sigma)\rangle$ is no longer isomorphic
to $\bm{Z}[\zeta_d]$.

When $d=1$, $\bm{Z}\pi/\langle\Phi_d(\sigma)\rangle$ is isomorphic
to $\bm{Z}\pi'$ where $\pi' \simeq C_2$. Since
$(M^*/\Phi_d(\sigma)M^*)_0$ is a torsion-free module over
$\bm{Z}\pi/\langle \Phi_d(\sigma)\rangle$ (where $M$ is an
invertible $\pi$-lattice), we may apply Step 5 of Case 1. Thus
there is a projective ideal $\c{A}$ over $\bm{Z}\pi' (\simeq
\bm{Z}\pi/\langle\Phi_d(\sigma)\rangle)$ such that
$[(M/\Phi_d(\sigma)M)_0]^{fl}=[\c{A}]^{fl}$. Note that $\c{A}$ and
$\bm{Z}\pi'$ are in the same genus as $\pi'$-lattices and also as
$\pi$-lattices. Apply Theorem \ref{t2.10} to find a projective
ideal $\c{B}$ over $\bm{Z}\pi$ such that $\c{A} \oplus\bm{Z}\pi
\simeq\bm{Z}\pi' \oplus\c{B}$. Hence
$[(M/\Phi_d(\sigma)M)_0]^{fl}=[\c{B}]^{fl}$. Done.

When $d \mid p^c$ and $d > 1$, we will show that
$\bm{Z}\pi/\langle\Phi_d(\sigma)\rangle$ is isomorphic to a
twisted group ring.

\medskip
Step 2. For any $d\mid p^c$ such that $d > 1$, write $\Lambda_d
:=\bm{Z}\pi/\langle\Phi_d(\sigma)\rangle =S_d\circ H$ where
$S_d=\bm{Z}[\zeta_d]$, $H=\langle\tau\rangle$ and $\zeta_d\in S_d$
is the image of $\sigma$ in
$\bm{Z}\pi/\langle\Phi_d(\sigma)\rangle$. Define
$R_d=S_d^{\langle\tau\rangle}$; note that $\tau\cdot
\zeta_d=\zeta_d^{-1}$.

The only prime ideal which ramifies in $\bm{Z}[\zeta_{p^c}]$ over $\bm{Z}$ is the prime ideal lying over $p\bm{Z}$ and
$p\bm{Z}[\zeta_{p^c}]=Q^{p^{c-1}(p-1)}$ where $Q=(1-\zeta_{p^c})\bm{Z}[\zeta_{p^c}]$ (see, for example, \cite[page 96]{CR1}).
Thus, for any $d\mid p^c$ (with $d>1$), $S_d$ is tamely ramified over $R_d$.

By Theorem \ref{t3.6}, $\Lambda_d$ is a left hereditary ring and the ambiguous ideals are isomorphic to $IS_d$ or $IQ_d$
where $I$ is an ideal in $R_d$ and $Q_d=Q\cap S_d$
where $Q=(1-\zeta_{p^c})\bm{Z}[\zeta_{p^c}]$

Note that $IS_d$ and $S_d$ belong to the same genus because they are locally isomorphic;
similarly for $IQ_d$ and $Q_d$.

\medskip
Step 3. For any $d\mid p^c$, applying Theorem \ref{t4.3}, we find
that $[\bm{Z}\pi/\langle\Phi_d(\sigma)\rangle]^{fl}=0$ because
$\bm{Z}\pi/\langle\sigma^{e'}-1\rangle \simeq \bm{Z}[\pi/\pi']$
where $\pi'=\langle \sigma^{e'}\rangle$ if $e'\mid d$. In
conclusion, $[\Lambda_d]^{fl}=0$.

On the other hand, if $d\mid p^c$, define $\pi''=\langle
\sigma^d,\tau\rangle$ and $N=\bm{Z}[\pi/\pi'']$. It is not
difficult to verify that $(N^*/\Phi_d(\sigma)N^*)_0\simeq Q_d$ as
$\pi$-lattices.

Apply Theorem \ref{t4.3} to $N=\bm{Z}[\pi/\pi'']$. For any $e'\mid
d$, since $N/(\sigma^{e'}-1)N\allowbreak \simeq
\bm{Z}[\pi/\pi']\otimes_{\bm{Z}\pi} \bm{Z}[\pi/\pi'']$ (where
$\pi'=\langle\sigma^{e'}\rangle$) is a permutation $\pi$-lattice
by \cite[Lemma 5.3]{Sw6}, it follows that $[Q_d]^{fl}=0$.

\medskip
Step 4. For any $d\mid p^c$ (with $d > 1$), we will prove in Step
6 that $\Lambda_d\simeq S_d\oplus Q_d$ as $\Lambda_d$-modules, and
hence as $\pi$-lattices.

Assume the above claim. By Step 3,
$[\Lambda_d]^{fl}=[Q_d]^{fl}=0$, it follows that $[S_d]^{fl}=0$.

Let $M$ be an invertible $\pi$-lattice. Apply Theorem \ref{t4.3}
as in Step 3 of Case 1. Since $(M^*/\Phi_d(\sigma)M^*)_0$ is a
torsion-free module over $\bm{Z}\pi/\langle
\Phi_d(\sigma)\rangle=\Lambda_d$, we may apply results in Step 2.
Thus $(M^*/\Phi_d(\sigma)M^*)_0$ is a direct sum of $S_d I_i$
($1\le i\le u$) and $Q_dI_j$ ($1\le j\le v$) where $I_i$, $I_j$
are ideals in $R_d$. In particular, $(M^*/\Phi_d(\sigma)M^*)_0$
and $S_d^{(u)}\oplus (Q_d)^{(v)}$ are in the same genus. Applying
Theorem \ref{t2.10}, we find a projective ideal $\c{A}_d$ over
$\bm{Z}\pi$ such that $(M^*/\Phi_d(\sigma)M^*)_0\oplus
\bm{Z}\pi\simeq (S_d^{(u)}\oplus(Q_d)^{(v)})\oplus\c{A}_d$. Thus
$[(M^*/\Phi_d(\sigma)M^*)_0]^{fl}=[\c{A}_d]^{fl}$.

\medskip
Step 5. From Step 1 and Step 4, we find $[M]^{fl}=\sum_{d\mid n}
[\c{A}_d]^{fl}$ (the reader may compare the present situation with
Step 3 of Case 1).

Use the same argument as in Step 4 of Case 1. It is easy to find a
projective ideal $\c{A}$ over $\bm{Z}\pi$ satisfying
$[M]^{fl}=[\c{A}]^{fl}$.

\medskip
Step 6. It remains to show that $\Lambda_d\simeq S_d\oplus Q_d$.

Recall that $\Lambda_d=S_d\circ H=S_d+S_d u_\tau$ such that
$u_\tau^2=1$, $u_\tau\cdot \alpha=\tau(\alpha)\cdot u_\tau$ for
any $\alpha\in S_d$. Remember that $\zeta_d$ is the image of
$\sigma$ in $\Lambda_d$, and $S_d=\bm{Z}[\zeta_d]$,
$Q_d=(\zeta_d-\zeta_d^{-1})S_d$.

Define $w=\zeta_d+\zeta_d^2+\ldots+\zeta_d^{(d-1)/2} \in
\Lambda_d$ and $e=- \tau(w)(1+u_{\tau}) \in \Lambda_d$. Note that
$w$, $\tau(w)$ are units in $S_d$ and $1+w+\tau(w)=0$.

It is not difficult to show that $e^2=e$. It follows that
$\Lambda_d=\Lambda_d e \oplus \Lambda_d (1-e)$.

Note that $\Lambda_d e=\Lambda_d(1+u_{\tau})=S_d(1+u_{\tau})\simeq
S_d$. On the other hand, from $1-e=-(1-u_{\tau})w$, we find that
$\Lambda_d (1-e)=S_d(1-u_{\tau})w\simeq Q_d$. Hence the result.

\bigskip
\textit{Subcase 2.2} There exist distinct odd prime numbers $p_1$ and $p_2$ such that $p_1p_2\mid m$.

Step 1. Let $m'$ be an odd integer such that $p_1p_2\mid m'$ for two
distinct prime numbers. Then $\bm{Z}[\zeta_{m'}]$ is unramified
over $\bm{Z}[\zeta_{m'}+\zeta_{m'}^{-1}]$ by \cite[page 16,
Proposition 2.15]{Wa}. This observation will play a crucial role
in the subsequent proof.

\medskip
Step 2.
Let $M$ be any invertible $\pi$-lattice.
We will find a projective ideal $\c{A}$ over $\bm{Z}\pi$ such that $[M]^{fl}=[\c{A}]^{fl}$.
The proof is similar to that of Subcase 2.1.

Apply Theorem \ref{t4.3}. We may reduce the problem to the case
$(M^*/\Phi_d(\sigma)M^*)_0$ for any $d\mid m$.

If $d=1$, then Step 5 of Case 1 takes care of this situation (see
Step 1 of Subcase 2.1).

If $d > 1$, note that $(M^*/\Phi_d(\sigma)M^*)_0$ is a
torsion-free module over
$\Lambda_d:=\bm{Z}\pi/\langle\Phi_d(\sigma)\rangle=S_d\circ H$
where $S_d=\bm{Z}[\zeta_d]$, $H=\langle\tau\rangle$ and $\zeta_d$
is the image of $\sigma$ in
$\bm{Z}\pi/\langle\Phi_d(\sigma)\rangle$. The group $H$ acts on
$S_d$ by $\tau\cdot \zeta_d=\zeta_d^{-1}$. Define
$R_d=S_d^{\langle\tau\rangle}=\bm{Z}[\zeta_d+\zeta_d^{-1}]$.

Suppose that $d$ has two distinct prime divisors. Then $S_d/R_d$
is unramified by Step 1. Thus $\Lambda_d$ is hereditary; moreover,
$S_dI$ are the only ambiguous ideals of $S_d$ (where $I$ runs over
some ideals in $R_d$) by Theorem \ref{t3.6}. We conclude that
$(M^*/\Phi_d(\sigma)M^*)_0$ and $S_d^{(u)}$ belong to the same
genus. It can be proved that $[S_d]^{fl}=0$ as in Step 3 and Step
6 of Subcase 2.1. Thus we may find a projective ideal $\c{A}_d$
over $\bm{Z}\pi$ satisfying
$[(M^*/\Phi_d(\sigma)M^*)_0]^{fl}=[\c{A}_d]^{fl}$ as Step 4 of
Subcase 2.1.

On the other hand if $d=p^c$ for some odd prime number $p$ and
some $c\ge 1$, then the ideal $(1-\zeta_{p^c})S_d$ is the only
ramified prime ideal of $S_d$ over $R_d$; moreover, $S_d$ is
tamely ramified over $R_d$. The remaining proof is the same as
that of Subcase 2.1. Done.

\bigskip
\begin{Case}{3}
$\pi=\langle\sigma,\rho,\tau:\rho^{q^f}=\sigma^m=\tau^2=1,\tau^{-1}\rho\tau=\rho,\tau^{-1}\sigma\tau=\sigma^{-1},\sigma\rho=\rho\sigma\rangle\simeq
C_{q^f}\times D_m$ where $q$ is an odd prime number, $f\ge 1$, $m$
is an odd integer $\ge 3$, $\gcd\{q,m\}=1$, and
$(\bm{Z}/q^f\bm{Z})^{\times}=\langle \bar{p}\rangle$ for any prime
divisor $p$ of $m$.
\end{Case}

We will solve this case by induction on the order of the group
$\pi$.

\medskip
Step 1. Suppose that $M$ be an invertible $\pi$-lattice. Apply
Theorem \ref{t4.4}.

When $d \mid m$ and $d \neq m$, the lattice
$(M^*/\Phi_d(\sigma)M^*)_0$ may be regarded as a lattice over
$\bm{Z}\pi/\langle \sigma^d \rangle$. Replacing this lattice by a
coflabby lattice (applying Lemma \ref{l2.6}) and using the
induction hypothesis, we find a projective ideal $\c{B}_d$ over
$\bm{Z}\pi/\langle \sigma^d \rangle$ such that
$[(M^*/\Phi_d(\sigma)M^*)_0]^{fl}=[\c{B}_d]^{fl}$. By Theorem
\ref{t2.10}, find a projective ideal $\c{A}_d$ over $\bm{Z}\pi$
such that $\c{B}_d \oplus \bm{Z}\pi \simeq\bm{Z}\pi/\langle
\sigma^d \rangle \oplus \c{A}_d$. Hence
$[(M^*/\Phi_d(\sigma)M^*)_0]^{fl}=[\c{A}_d]^{fl}$. Similarly, for
the lattice $M^{\prime}$ in Theorem \ref{t4.4}, by induction on
the order of the group $\pi$, we also can find a projective ideal
$\c{A}^{\prime}$ over $\bm{Z}\pi$ such that
$[M^{\prime}]^{fl}=[\c{A}^{\prime}]^{fl}$.

By Theorem \ref{t4.4}, $[M]^{fl} = [N]^{fl} + [\c{A}]^{fl}$ where
$N:=(M^*/\langle \Phi_{q^f}(\rho), \Phi_m(\sigma)\rangle M^*)_0$
and $\c{A}$ is a projective ideal over $\bm{Z}\pi$.

In summary, we have reduced the question from $[M]^{fl}$ to
$[N]^{fl}$ where $N$ is a torsion-free module over $\Lambda$ with
$\Lambda:=\bm{Z}\pi/\langle \Phi_{q^f}(\rho),
\Phi_m(\sigma)\rangle$.

\medskip

Step 2. Denote by $\zeta_{q^f}$ and $\zeta_m$ the images of $\rho$
and $\sigma$ in $\Lambda$ respectively. Note that $\zeta_{q^f}$
(resp. $\zeta_m$) is a primitive $q^f$-th root of unity (resp. a
primitive $m$-th root of unity). In the sequel we will write
$\Lambda=\bm{Z}[\zeta_{q^f},\zeta_m,\tau]$. Note that
$\tau^{-1}\cdot \zeta_{q^f}\cdot \tau=\zeta_{q^f}$,
$\tau^{-1}\zeta_m\tau=\zeta_m^{-1}$ and $\bm{Z}[\zeta_{q^f},
\zeta_m]=\bm{Z}[\zeta_{q^f}]\otimes_{\bm{Z}}\bm{Z}[\zeta_m]$,
because $\gcd\{q,m \}=1$.

$\Lambda$ can be written as a twisted group ring: $\Lambda=S\circ
H$ where $S=\bm{Z}[\zeta_{q^f}, \zeta_m]$ and
$H=\langle\tau\rangle$. Define
$R=S^{\langle\tau\rangle}=\bm{Z}[\zeta_{q^f}][\zeta_m+\zeta_m^{-1}]$.

Assume that $p_1p_2\mid m$ for two distinct odd prime numbers (the
case $m=p^c$ for an odd prime number will be considered in Step
3).

Then $S/R$ is unramified as in Step 1 of Subcase 2.2. Hence
$\Lambda$ is hereditary; moreover, $SI$ are the only ambiguous
ideals of $S$ (where $I$ runs over some ideals in $R$) by Theorem
\ref{t3.6}. It follows that $N$ and $S^{(u)}$ are in the same
genus for some integer $u$. The remaining proof is the same as in
the proof of Step 2 of Subcase 2.2.

\medskip
Step 3. Suppose that $m=p^c$ for some odd prime number $p$ and
some $c\ge 1$.

Use the assumption that $(\bm{Z}/q^f \bm{Z})^{\times}=\langle
\bar{p}\rangle$. This assumption is equivalent to the fact that
$p\cdot\bm{Z}[\zeta_{q^f}]$ is a prime ideal in
$\bm{Z}[\zeta_{q^f}]$ (just think of the decomposition of $p\cdot
\bm{Z}[\zeta_{q^f}]$ and the Frobenius automorphism associated to
a prime ideal $Q$ of $\bm{Z}[\zeta_{q^f}]$ such that $p\in Q$).

Return to $\Lambda=S\circ H$ with $H=\langle\tau\rangle$,
$R=S^{\langle\tau\rangle}$.

We claim that the principal ideal $\langle
1-\zeta_{p^c}\rangle=(1-\zeta_{p^c}) S$ is a prime ideal. For,
$S/\langle
1-\zeta_{p^c}\rangle=\bm{Z}[\zeta_{p^c},\zeta_{q^f}]/\langle
1-\zeta_{p^c}\rangle \simeq \bm{F}_p[X]/\Phi_{q^f}(X)$ is a field
because $p\cdot\bm{Z}[\zeta_{q^f}]$ is a prime ideal (see, for
example, \cite[page 15, Proposition 2.14]{Wa}).

Note that $\Lambda$ is tamely ramified. Define $Q=\langle
1-\zeta_{p^c}\rangle$. As in Subcase 2.1, it can be shown that
$[\Lambda]^{fl}=[S]^{fl}=[Q]^{fl}=0$. The remaining proof is the
same and is omitted.

\medskip
Step 4. Finally we remark that we cannot replace the prime power
$q^f$ in the assumption by an odd integer $n$ with
$\gcd\{n,m\}=1$, because $\bm{F}_p[X]/\Phi_n(X)$ is not an
integral domain if $q_1q_2\mid n$ for two distinct odd prime
numbers $q_1$, $q_2$ different from $p$. Thus $\langle
1-\zeta_{p^c}\rangle$ is not a prime ideal of $S$ if
$\bm{F}_p[X]/\Phi_n(X)$ is not a field. See \cite[page 188]{EM4}
for further investigation.

\bigskip
\begin{Case}{4}
$\pi=\langle\sigma,\tau:\sigma^{2m}=\tau^4=1,\sigma^m=\tau^2,\tau^{-1}\sigma\tau=\sigma^{-1}\rangle\simeq
Q_{4m}$ where $m$ is an odd integer $\ge 3$ such that $p\equiv 3
\pmod{4}$ for any prime divisor $p$ of $m$.
\end{Case}

\medskip
Step 1. We follow the method in Cases 1,2,3.

Let $M$ be an invertible $\pi$-lattice.

Suppose $(\sigma^n-1) \cdot M=0$ for some $n \mid 2m$. If $n=1$ or
$2$, we may regard $M$ as a lattice over $\bm{Z}\pi''$ where
$\pi''=\pi/\langle \sigma^n \rangle$ and $\pi'' \simeq C_2$ or
$C_4$. Apply Case 1.

From now on, we assume that $n \ge 3$ and apply Theorem
\ref{t4.3}.

For any $d\mid n$, consider $(M^*/\Phi_d(\sigma)M^*)_0$ which is a
module over $\Lambda_d:=\bm{Z}\pi/\langle\Phi_d(\sigma)\rangle$.

If $d=1$, then $\Lambda_d \simeq \bm{Z}\pi'$ where $\pi' \simeq
C_2$. If $d=2$, then $\Lambda_d \simeq \bm{Z}[\sqrt{-1}]$; thus
$\bm{Z}\pi/\langle \sigma^2 \rangle \to \Lambda_d$ is surjective
(note that $\bm{Z}\pi/\langle \sigma^2 \rangle \simeq \bm{Z}\pi'$
where $\pi' \simeq C_4$). Apply the methods in Step 1 of Subcase
2.1 and Step 2 of Case 3 to settle these two cases.

Now consider the case $d \ge 3$.

Write $\zeta_d$ and $u_\tau$ to be the images of $\sigma$ and
$\tau$ in $\Lambda_d$ respectively. Define $S_d=\bm{Z}[\zeta_d]$.

If $d \mid m$, from $\tau^2=\sigma^m$, we find $u_\tau^2=1$ in
$S_d$. Thus $\Lambda_d \simeq S_d \circ H$ where
$H=\langle\tau'\rangle \simeq C_2$. The proof is the same as Case
2.

If $d\mid 2m$, $d \nmid m$ and $d\ge 3$, since $\tau^2=\sigma^m$,
$(\sigma^m)^2=\sigma^{2m}=1$, we find $u_\tau^2=-1$ in $S_d$.

Thus $\Lambda_d=S_d+S_d\cdot u_\tau$ where $u_\tau^2=-1$. Note
that $\Lambda_d$ may be regarded as a crossed-product order
$\Lambda_d=(S_d\circ H)_f$ where $H=\langle\tau'\rangle\simeq
C_2$, $f:H\times H\to U(S_d)$ is defined as
$f(1,1)=f(1,\tau')=f(\tau',1)=1$, $f(\tau',\tau')=-1$. Define
$R_d=S_d^{\langle\tau'\rangle}$ where $\tau'\cdot
\zeta_d=\zeta_d^{-1}$.

\medskip
Step 2. From now on till the end of the proof, we assume $d\mid
2m$, $d \nmid m$ and $d\ge 3$. In Step 3 and Step 4, we will show
that $\Lambda_d$ is a maximal $R_d$-order in $(L\circ H)_f$ where
$L=\bm{Q}(\zeta_d)$.

Assume this fact. Then $\Lambda_d$ is hereditary and any
$\Lambda_d$-lattice is a direct sum of indecomposable projective
$\Lambda_d$-modules by Theorem \ref{t5.2}.

Let $K$ be the quotient field of $R_d$. Since $K \Lambda_d$ is a
central simple $K$-algebra of degree 2, it is either a central
division $K$-algebra or is isomorphic to $M_2(K)$. We will show
that $K \Lambda_d$ is a central division $K$-algebra. Otherwise,
the $2$-cocycle $f$ is a $2$-coboundary. Equivalently, $-1$
belongs to the image of the norm map from $L$ to $K$. This implies
that there exist $a, b \in K$ satisfying that $-1= a^2+ (\zeta_d +
\zeta_d^{-1})ab + b^2$, which is impossible because $K$ is a real
field.

Since $K \Lambda_d$ is a division ring, it is obvious that
$\Lambda_d$ is an indecomposable projective $\Lambda_d$-module by
Theorem \ref{t5.2}.

We may prove $[\Lambda_d]^{fl}=0$ as in Case 2. By Theorem
\ref{t5.2}, $(M^*/\Phi_d(\sigma)M^*)_0$ is in the same genus as
$\Lambda_d^{(u)}$. Apply Theorem \ref{t2.10} as in Case 2. Thus
there is a projective ideal $\c{A}_d$ over $\bm{Z}\pi$ such that
$[(M^*/\Phi_d(\sigma)M^*)_0]^{fl}=[\c{A}_d]^{fl}$. The remaining
proof is the same as before.

\medskip
Step 3. We will show that $\Lambda_d$ is a maximal order if
$p_1p_2\mid d$ for two distinct odd prime numbers $p_1$ and $p_2$.
The situation when $d=p^c$ or $2p^c$ (where $p$ is an odd prime
number and $c\ge 1$) will be taken care of in Step 4.

Since $p_1p_2\mid d$, $S_d$ is unramified over $R_d$. Thus $S_d$
is a Galois extension of $R_d$ relative to
$H=\langle\tau'\rangle\simeq C_2$ in the sense of \cite[pages
395--402]{AG2}. Thus $\Lambda_d=(S_d\circ H)_f$ is a central
separable $R_d$-algebra by \cite[page 402, Theorem A.12]{AG2}.
Hence $\Lambda_d$ is a maximal order in $K\Lambda_d$ (where $K$ is
the quotient field of $R_d$) by \cite[page 386, Proposition
7.1]{AG2}.

\medskip
Step 4. When $d=2p^c$ for some odd prime number $p$, we will show
that $\Lambda_d$ is a maximal order.

Note that $\zeta_{2p^c}=-\zeta_{p^c}$. For simplicity, we write
$\Lambda$ for $\Lambda_d$ throughout this step, i.e.\
$\Lambda=S+S\cdot u_\tau$ where $S=\bm{Z}[\zeta_{p^c}]$,
$\tau\cdot \zeta_{p^c}=\zeta_{p^c}^{-1}$, $u_\tau^2=-1$ and
$R=S^{\langle\tau\rangle}=\bm{Z}[\zeta_{p^c}+\zeta_{p^c}^{-1}]$.

Note that to be a maximal order is a local property \cite[page
132, Corollary 11.2]{Re1}. We will check it at all localizations
of $R$.

Let $P'$ be a prime ideal of $R$ such that $p\notin P'$. Write
$\Lambda_{P'}$, $S_{P'}$ and $R_{P'}$ for the localizations of
$\Lambda$, $S$ and $R$ at $P'$. It follows that $S_{P'}$ is
unramified over $R_{P'}$. Hence $\Lambda_{P'}$ is a central
separable algebra as in Step 3. Thus $\Lambda_{P'}$ is a maximal
order.

On the other hand, let $P$ be the prime ideal of $R$ with $p\in
P$. We will show that $\Lambda_P$ is also a maximal order.

Since $S_P$ is tamely ramified over $R_P$, $\Lambda_p$ is a
hereditary order by Theorem \ref{t3.2}.

We will apply Theorem \ref{t5.4} to show that $\Lambda_P$ is a
maximal order. Let $K$ be the quotient field of $R$. It is clear
that $K\Lambda=K\Lambda_P$ is a central simple $K$-algebra. It
remains to verify that $\fn{rad}(\Lambda_P)$ is a maximal
two-sided ideal of $\Lambda_P$ where $\fn{rad}(\Lambda_P)$ is the
Jacobson radical of $\Lambda_P$.

Consider $R_P\subset S_P\subset \Lambda_P$. It is clear that
$\fn{rad}(S_P)=(1-\zeta_{p^c}) S_P$. Thus
$(1-\zeta_{p^c})(1-\zeta_{p^c}^{-1}) \in R_P \cap \fn{rad}(S_P) =
\fn{rad}(R_P)$, since $S_P$ is integral over $R_P$. It follows
that $(1-\zeta_{p^c})(1-\zeta_{p^c}^{-1})\in \fn{rad}(\Lambda_P)$
because $\fn{rad}(R_P)\Lambda_P \subset \fn{rad}(\Lambda_P)$ by
\cite[page 82, Theorem 6.15]{Re1}. On the other hand,
$(1-\zeta_{p^c})(1-\zeta_{p^c}^{-1})=-\zeta_{p^c}^{-1}(1-\zeta_{p^c})^2$.
We find that $(1-\zeta_{p^c})^2 \in \fn{rad}(\Lambda_P)$. Hence
$1-\zeta_{p^c} \in \fn{rad}(\Lambda_P)$ (because, for any simple
module $M$ over $\Lambda_P$, $(1-\zeta_{p^c})^2 M=0$ implies
$(1-\zeta_{p^c}) M=0$). We conclude that the two-sided ideal
$\langle 1-\zeta_{p^c}\rangle=(1-\zeta_{p^c})\Lambda_P$ is
contained in $\fn{rad}(\Lambda_P)$.

We will show that $\Lambda_P/\langle 1-\zeta_{p^c}\rangle$ is a
field. Once it is proved, then $\langle
1-\zeta_{p^c}\rangle=\fn{rad}(\Lambda_P)$ and
$\Lambda_P/\fn{rad}(\Lambda_P)$ is a field. And therefore
$\fn{rad}(\Lambda_P)$ is a maximal two-sided ideal of $\Lambda_P$.

It remains to show that $\Lambda_P/\langle 1-\zeta_{p^c}\rangle$ is a field.
Note that $p\equiv 3 \pmod{4}$ by assumption;
this is equivalent to $\bm{F}_p[X]/\langle X^2+1\rangle$ is a field.

Note that $p\in \langle 1-\zeta_{p^c}\rangle$.
We have $u_\tau\cdot \zeta_{p^c}\cdot u_\tau^{-1}=\zeta_{p^c}^{-1}$.
However, in $\Lambda_d/\langle 1-\zeta_{p^c}\rangle$ we find that $\bar{u}_\tau\cdot \bar{\zeta}_{p^c}=\bar{\zeta}_{p^c}^{-1}\bar{u}_\tau %
=\bar{\zeta}_{p^c}\cdot \bar{u}_\tau+(\bar{\zeta}_{p^c}^{-1}-\bar{\zeta}_{p^c})\bar{u}_\tau=\bar{\zeta}_{p^c}\cdot\bar{u}_\tau$ because
$\zeta_{p^c}^{-1}-\zeta_{p^c}=\zeta_{p^c}^{-1}(1+\zeta_{p^c})(1-\zeta_{p^c})\in \langle 1-\zeta_{p^c}\rangle$.

Thus $\Lambda_P/\langle 1-\zeta_{p^c}\rangle$ is a commutative
ring; furthermore, $\Lambda_P/\langle 1-\zeta_{p^c}\rangle \simeq
\bm{F}_p[\zeta_{p^c},u_\tau]/\langle
1-\zeta_{p^c}\rangle\allowbreak\simeq \bm{F}_p[X]/(X^2+1)$ which
is a field by assumption.

A final remark. Since $\Lambda_P/\fn{rad}(\Lambda_P)$ is a field,
it follows that $K\Lambda=K\Lambda_P$ is a division ring by
\cite[page 189, Theorem 21.6]{Re1}, which has been proved in Step
2.
\end{proof}

\bigskip
We explain briefly the strategy of the proof of Theorem
\ref{t1.4}. In \cite{EM4}, this theorem is proved by showing
$(2)\Rightarrow (3)\Rightarrow (1') \Rightarrow (1)$ where $(1')$
is given as follows.

\begin{enumerate}
\item[$(1')$] $\pi$ is isomorphic to
\begin{enumerate}
\item[(i$'$)] a cyclic group, or \item[(ii$'$)] $C\times H$ where
$C=\langle \sigma\rangle$ is a cyclic group of order $n$,
$H=\langle\rho,\tau:\rho^m=\tau^{2^d}=1$,
$\tau^{-1}\rho\tau=\rho^{-1}\rangle$ such that $n$ and $m$ are odd
positive integers, $\gcd\{n,m\}=1$, $d\ge 1$, and, for any prime
divisor $p$ of $m$, the principal ideal $p\bm{Z}[\zeta_{n\cdot
2^d}]$ is a prime ideal in $\bm{Z}[\zeta_{n\cdot 2^d}]$.
\end{enumerate}
\end{enumerate}

The implication of $(2)\Rightarrow (3)$ is easy, because
$C(\bm{Z}\pi)\to C(\Omega_{\bm{Z}\pi})$ is surjective (see
Definition \ref{d2.12}).

The proof of $(3)\Rightarrow (1')$ was given in \cite[pages
97--98, pages 188--189]{EM4}. We will not repeat the proof.

For the remaining part of this section, we will give a proof of
$(1)\Leftrightarrow (1')$.

\bigskip
\begin{proof}[Proof of $(1)\Leftrightarrow (1')$] ~ \par
For the proof of $(1)\Rightarrow (1')$, if $\pi\simeq C_n$ or
$D_m$, it is trivial to see that $\pi$ belongs to the class
described in $(1')$. If
$\pi=Q_{4m}=\langle\sigma,\tau:\sigma^{2m}=\tau^4=1,\sigma^m=\tau^2,\tau^{-1}\sigma\tau=\sigma^{-1}\rangle$,
define $\rho=\sigma^2$. Then
$\pi=\langle\rho,\tau:\rho^m=\tau^4=1,\tau^{-1}\rho\tau=\rho^{-1}\rangle$.
The assumption on the prime divisor $p$ of $m$ in (iv) of (1) is
equivalent to $p\cdot \bm{Z}[\sqrt{-1}]$ is a prime ideal in
$\bm{Z}[\sqrt{-1}]$ (we have $d=2$, $n=1$ for (ii$'$) of $(1')$ in
this situation).

Now consider the case $\pi=C_{q^f}\times D_m$ in (iii) of (1).
Take $n=q^f$ in (ii$'$) of $(1')$. If $p$ is a prime divisor of
$m$, since $p\ne q$, the prime number $p$ is unramified in
$\bm{Q}(\zeta_{q^f})$. Let $P$ be a prime ideal of
$\bm{Z}[\zeta_{q^f}]$ with $p\in P$. Since $(\bm{Z}/q^f
\bm{Z})^{\times}=\langle\bar{p}\rangle$, it follows that
$\fn{Gal}(\bm{Q}(\zeta_{q^f})/\bm{Q})=\langle\varphi\rangle$ where
$\varphi(\zeta_{q^f})=(\zeta_{q^f})^p$. On the other hand,
$\varphi$ induces the Frobenius automorphism of
$\bm{Z}[\zeta_{q^f}]/P$ over $\bm{Z}/p\bm{Z}$. It follows that
$P=p\bm{Z}[\zeta_{q^f}]$ and thus $p$ remains prime in
$\bm{Z}[\zeta_{q^f}]$ as expected.

For the proof of $(1')\Rightarrow (1)$, we first note that, if $n$
is odd and $d\ge 1$, then $(\bm{Z}/n\cdot 2^d \bm{Z})^{\times}$ is
a cyclic group if and only if $(n,2^d)=(1,2)$, (1,4), or $(q^f,2)$
where $q$ is some odd prime number.

On the other hand, note that $\fn{Gal}(\bm{Q}(\zeta_{n\cdot
2^d})/\bm{Q})\simeq (\bm{Z}/n\cdot 2^d \bm{Z})^{\times}$. If a
prime number $p$ remains prime in $\bm{Z}[\zeta_{n\cdot 2^d}]$,
then $[\bm{Z}[\zeta_{n\cdot 2^d}]/p\bm{Z}[\zeta_{n\cdot
2^d}]:\bm{Z}/p\bm{Z}]=|(\bm{Z}/n\cdot 2^d \bm{Z})^{\times}|$. Thus
the Frobenius automorphism of $p$ generates
$\fn{Gal}(\bm{Q}(\zeta_{n\cdot 2d})/\bm{Q})$. Hence
$(\bm{Z}/n\cdot 2^d \bm{Z})^{\times}$ is cyclic and
$(n,2^d)=(1,2)$, (1,4), or $(q^f,2)$.

When $(n,2^d)=(1,2)$, the group $\pi$ in (ii$'$) of $(1')$ is
isomorphic to $D_m$.

When $(n,2^d)=(1,4)$, the group $\pi$ in (ii$'$) of $(1')$ is
isomorphic to $Q_{4m}$ such that $p\equiv 3 \pmod{4}$ for any
prime divisor $p$ of $m$.

When $(n,2^d)=(q^f,2)$, the group $\pi$ in (ii$'$) of $(1')$ is
isomorphic to $C_{q^f}\times D_m$. Since every prime divisor $p$
of $m$ remains prime in $\bm{Z}[\zeta_{q^f}]$, write
$P=p\bm{Z}[\zeta_{q^f}]$ the prime ideal of $\bm{Z}[\zeta_{q^f}]$.
Then the Frobenius automorphism of $\bm{Z}[\zeta_{q^f}]/P$
generates $\fn{Gal}(\bm{Q}(\zeta_{q^f})/\bm{Q})\simeq
(\bm{Z}/q^f\bm{Z})^{\times}$. Thus
$(\bm{Z}/q^f\bm{Z})^{\times}=\langle\bar{p}\rangle$.
\end{proof}

%------------------------------------------S6
\section{The maximal orders}

Because of Theorem \ref{t1.4}, we will determine
$C(\Omega_{\bm{Z}\pi})$ when $\pi=C_n$, $D_n$, $C_n\times D_m$,
$Q_{4n}$. Given a finite group $\pi$, there may be more than one
maximal order containing $\bm{Z}\pi$. For our purpose, it is
enough to select just one of them.

%------------------d6.1
\begin{defn} \label{d6.10}
Let $K$ be an algebraic number field, $A$ be a cental simple
$K$-algebra, $v$ be a place of $K$ (finite or infinite). We say
that $A$ ramifies at $v$ if $[K_v\otimes_K A]\ne 0$ in the Brauer
group of $K_v$ where $K_v$ is the completion of $K$ at $v$
\cite[page 272]{Re1}.
\end{defn}

%------------------d6.2
\begin{defn} \label{d6.11}
Let $R$ be a Dedekind domain whose quotient field $K$ is an
algebraic number field, and let $A$ be a central simple
$K$-algebra. Define $I(R)$ to be the multiplicative group of fractional $R$-ideals in $K$, and define $P(R)=\{R\alpha:\alpha\in
K\backslash \{0\}\}$. Recall that $C(R)=I(R)/P(R)$.

Define $S$ to be the set of all infinite places of $K$ ramified in
$A$. Define $P_A(R)=\{R\alpha:\alpha\in K\backslash \{0\},
\alpha_v>0$ for all $v$ in $S\}$, the principal ray group (mod
$S$). The ray class group (mod $S$), denoted by $C_A(R)$, is
defined as $C_A(R)=I(R)/P_A(R)$ \cite[page 309]{Re1}. If $S$ is
the empty set, then $C(R) \simeq C_A(R)$.

In general, the kernel of the surjective map $C_A (R) \to C(R)$ is
the group $P(R)/P_A (R)$ which may be computed by the exact
sequence $U(R) \to D \to P(R)/P_A(R) \to 0$, where $U(R)$ is the
group of units of $R$, $D= \prod_{v \in S}
K_v^{\times}/(K_v^{\times})^{+}$; note that $K_v^{\times}$ is the
multiplicative group of non-zero elements in $K_v$,
$(K_v^{\times})^{+}$ is the group of positive elements of
$K_v^{\times}$, and $K_v^{\times}/(K_v^{\times})^{+} \simeq
\bm{Z}/2\bm{Z}$. The map $U(R) \to D$ is defined by $\alpha \mapsto
(\ldots, \bar{\alpha}_v, \ldots)$ where $\alpha_v$ is the image of
$\alpha$ in $K_v$; the map of $D$ to $P(R)/P_A(R)$ can be found in
\cite[page 139]{Sw3}.

\end{defn}

%------------------t6.3
\begin{theorem}[{Swan \cite[page 313, Theorem 35.14]{Re1}}] \label{t6.12}
Let the notations be the same as in Definition \ref{d6.11}. Let
$\Lambda$ be a maximal $R$-order in $A$. Then $C(\Lambda)\simeq
C_A(R)$ under the reduced norm map.
\end{theorem}

%------------------t6.4
\begin{theorem} \label{t6.13}
Let $\pi$ be a group, $\Omega_{\bm{Z}\pi}$ be a maximal
$\bm{Z}$-order in $\bm{Q}\pi$ containing $\bm{Z}\pi$.
\begin{enumerate}
\item[{\rm (1)}] If $\pi=C_n$, then
\[
C(\Omega_{\bm{Z}\pi})\simeq \oplus_{d\mid n} C(\bm{Z}[\zeta_d]).
\]
\item[{\rm (2)}] If $\pi=D_n$ where $n$ is an integer $\ge 2$,
then
\[
C(\Omega_{\bm{Z}\pi})\simeq \oplus_{d\mid n}
C(\bm{Z}[\zeta_d+\zeta_d^{-1}]).
\]
\item[{\rm (3)}] If $\pi=C_n\times D_m$ where $\gcd\{n,m\}=1$ and
$m$ is an integer $\ge 2$. For any $d \mid nm$, write $d=d_1d_2$
where $d_1 \mid n$, $d_2 \mid m$.

If $m$ is odd, then
\[
C(\Omega_{\bm{Z}\pi})\simeq (\oplus_{d \mid n}
C(\bm{Z}[\zeta_d])^{(2)})\oplus (\oplus_{d \mid nm \atop d\nmid n}
C(\bm{Z}[\zeta_{d_1},\zeta_{d_2}+\zeta_{d_2}^{-1}])).
\]

If $m$ is even, then
\[
C(\Omega_{\bm{Z}\pi})\simeq (\oplus_{d \mid n}
C(\bm{Z}[\zeta_d])^{(4)})\oplus (\oplus_{d \mid nm\atop d\nmid
n,d\ge 3} C(\bm{Z}[\zeta_{d_1},\zeta_{d_2}+\zeta_{d_2}^{-1}])).
\]
\item[{\rm (4)}] If $\pi=Q_{4n}$ where $n$ is an integer $\ge 2$,
then
\[
C(\Omega_{\bm{Z}\pi})\simeq \oplus_{d\mid n}
(C(\bm{Z}[\zeta_d+\zeta_d^{-1}]))\oplus (\oplus_{d\mid 2n \atop
d\nmid n, d \ge 3} C_{A_d}(R_d))
\]
where $C_{A_d}(R_d)$ is defined in Definition \ref{d6.11} with
$R_d=\bm{Z}[\zeta_d+\zeta_d^{-1}]$,
$K_d=\bm{Q}(\zeta_d+\zeta_d^{-1})$, $L_d=\bm{Q}(\zeta_d)$, and
$A_d$ is the central simple $K_d$-algebra defined by $A_d=L_d +
L_d \cdot u$ with $u^2=-1$, $u \alpha = \tau(\alpha) u$ for any
$\alpha \in L_d$, $\tau$ acts on $L_d$ by $\tau \cdot
\zeta_d=\zeta_d^{-1}$.
\end{enumerate}
\end{theorem}

\begin{proof}
Case 1. $\pi=C_n$.

Write $\bm{Z}\pi=\bm{Z}[X]/\langle X^n-1\rangle$. Note that
$\bm{Z}[X]/\langle X^n-1\rangle \hookrightarrow \prod_{d\mid n}
\bm{Z}[X]/\langle\Phi_d(X)\rangle \hookrightarrow
\bm{Q}[X]/\langle X^n-1\rangle$. Hence $\Omega_{\bm{Z}\pi} \simeq
\prod_{d\mid n} \bm{Z}[\zeta_d]$.

\bigskip
Case 2.
$\pi=\langle\sigma,\tau:\sigma^n=\tau^2=1,\tau^{-1}\sigma\tau=\sigma^{-1}\rangle\simeq
D_n$.

For any monic polynomial $f(X)\mid X^n-1$, the right ideal
$f(\sigma)\cdot \bm{Z}\pi$ is a two-sided ideal in $\bm{Z}\pi$
(see Step 1 in the proof of Theorem \ref{t4.3}); write it as
$\langle f(\sigma)\rangle$. By abusing the notation, we will also
write $\langle f(\sigma)\rangle$ as the two-sided ideal in
$\bm{Q}\pi$ generated by $f(\sigma)$.

It is not difficult to verify that $\bm{Z}\pi \hookrightarrow
\prod_{d\mid n} \bm{Z}\pi/\langle \Phi_d(\sigma)\rangle$ and
$\bm{Q}\pi =\prod_{d\mid n} \bm{Q}\pi
/\langle\Phi_d(\sigma)\rangle$.

When $d=1$ or $2$, $\bm{Q}\pi/\langle\Phi_d(\sigma)\rangle \simeq
\bm{Q}C_2 \simeq \bm{Q} \times \bm{Q}$. It follows that the
maximal $\bm{Z}$-order containing $\bm{Z}\pi/\langle
\Phi_d(\sigma)\rangle$ is $\bm{Z} \times \bm{Z}$. Thus the factor
$\bm{Q}\pi/\langle\Phi_d(\sigma)\rangle$ (for $d=1$ or $2$) has no
contribution to $C(\Omega_{\bm{Z}\pi})$.

When $d\ge 3$, write
$\Lambda_d:=\bm{Z}\pi/\langle\Phi_d(\sigma)\rangle$,
$A_d:=\bm{Q}\pi/\langle\Phi_d(\sigma)\rangle$. We will show that
$A_d$ is a simple algebra whose center $K_d$ is an algebraic
number field. Let $\Gamma_d$ be a maximal $\bm{Z}$-order with
$\Lambda_d \subset \Gamma_d \subset A_d$. It follows that
$C(\Omega_{\bm{Z}\pi})\simeq \oplus_{d\mid n \atop d \ge 3}
C(\Gamma_d)$. We will show that $C(\Gamma_d) \simeq
C(\bm{Z}[\zeta_d + \zeta_d^{-1}]$.

\medskip
When $d\ge 3$, $A_d \simeq L_d \circ H$ where
$L_d=\bm{Q}(\zeta_d)$, $H=\langle\tau\rangle \simeq C_2$ and
$\tau(\zeta_d)=\zeta_d^{-1}$. Define
$K_d=L_d^{\langle\tau\rangle}=\bm{Q}(\zeta_d+\zeta_d^{-1})$. Note
that $L_d\circ H \simeq M_2(K_d)$ is a matrix ring over $K_d$.
Thus $K_d$ has no ramified infinite place for $A_d$. Let $R_d$ be
the ring of algebraic integers of $K_d$. By Theorem \ref{t6.12},
$C(\Gamma_d)\simeq C(R_d)$.

\medskip

\bigskip
Case 3. $\pi=\langle\sigma',\rho,\tau: \sigma'^n=\rho^m=\tau^2=1,\tau^{-1}\sigma'\tau=\sigma', %
\tau^{-1}\rho\tau=\rho^{-1},\sigma'\rho=\rho\sigma'\rangle\simeq
C_n\times D_m$.

Define $\sigma=\sigma'\rho$ and consider
$\Lambda_d:=\bm{Z}\pi/\langle\Phi_d(\sigma)\rangle$,
$A_d:=\bm{Q}\pi/\langle\Phi_d(\sigma)\rangle$ where $d\mid nm$ as
in Case 2. The proof is almost the same as that in Case 2. Write
$d=d_1d_2$ where $d_1\mid n$ and $d_2\mid m$.

\medskip
Case 3.1 $m$ is odd.

If $d \mid n$, then $\Lambda_d$ is isomorphic to the group ring of
$C_2=\langle\tau\rangle$ over $S_d=\bm{Z}[\zeta_d]$ (see Step 2 of
Case 3 in the proof of Theorem \ref{t1.4}). A maximal order
containing $\Lambda_d$ is $S_d\cdot(1+\tau)/2\oplus S_d\cdot
(1-\tau)/2$. Thus the contribution of $\Lambda_d$ to
$C(\Omega_{\bm{Z}\pi})$ is $C(\bm{Z}[\zeta_d])\oplus
C(\bm{Z}[\zeta_d])$.

\medskip
Now consider $\Lambda_d$ where $d \nmid n$. Then $A_d\simeq
L_d\circ H$ where we keep the notations of $H$, $L_d$ as in Case
2, but $K_d$ should be replaced by
$K_d=\bm{Q}(\zeta_{d_1},\zeta_{d_2}+\zeta_{d_2}^{-1})$.

As in Case 2, $C(\Lambda_d)\simeq C(R_d)$.

\medskip
Case 3.2 $m$ is even.

The proof is almost the same as Case 3.1. Consider separately the
three situations $d \mid n$, $d \nmid n$ (but $d_2 \neq 2$) and
$d_2=2$.

For the last situation, if $d_2=2$, then $\Lambda_d$ is isomorphic to
the group ring of $C_2=\langle\tau\rangle$ over
$S_d=\bm{Z}[\zeta_d]=\bm{Z}[\zeta_{d/2}]$. This explains the reason
why there is an extra summand $C(\bm{Z}[\zeta_{d/2}])^{(2)}$.

\bigskip
Case 4.
$\pi=\langle\sigma,\tau:\sigma^{2n}=\tau^4=1,\sigma^n=\tau^2,\tau^{-1}\sigma\tau=\sigma^{-1}\rangle\simeq
Q_{4n}$.

As in Case 2, for $d\mid 2n$, consider
$\Lambda_d:=\bm{Z}\pi/\langle\Phi_d(\sigma)\rangle$.

If $d=1$ or $2$, $\Lambda_d$ has no contribution to
$C(\Omega_{\bm{Z}\pi})$ (note that, when $d=2$ and $n$ is odd,
$\bm{Z}\pi/\langle\Phi_d(\sigma)\rangle \simeq \bm{Z}[\sqrt{-1}])$
.

\medskip
It remains to consider $\Lambda_d$ where $d\ge 3$.

Define $A_d:=\bm{Q}\pi/\langle\Phi_d(\sigma)\rangle$ and
$\Gamma_d$ a maximal $\bm{Z}$-order with $\Lambda_d \subset
\Gamma_d \subset A_d$. Write $\zeta_d$ for the image of $\sigma$
in $\Lambda_d$ and in $A_d$, and $\tau^{\prime}$ for the image of
$\tau$.

If $d\mid n$ and $d\ge 3$, the relation $\tau^2=\zeta_d^n=1$
entails the consequence $A_d$ is a matrix ring over $K_d$ as in
Case 2. Thus, as before, $C(\Gamma_d)\simeq C(R_d)$ where
$R_d=\bm{Z}[\zeta_d + \zeta_d^{-1}]$.

\medskip
It remains to consider the case $d\mid 2n$, $d\nmid n$ and $d\ge 3$.

In $A_d$, we have $\tau^{\prime \, 2}=-1$. It follows that
$\Lambda_d=(S_d\circ H)_f$ where the notation is the same as in
Step 1 of Case 4 in the proof of Theorem \ref{t1.4}. Define
$R_d=S_d^{\langle\tau'\rangle}=\bm{Z}[\zeta_d+\zeta_d^{-1}]$ and
$K_d=\bm{Q}(\zeta_d+\zeta_d^{-1})$.

Note that $A_d$ is a central division $K_d$-algebra with
$[A_d:K_d]=4$; in fact, it is a totally definite quaternion
algebra by Step 2 of Case 4 in the proof of Theorem \ref{t1.4}. By
Theorem \ref{t6.12}, $C(\Lambda)\simeq C_{A_d}(R_d)$.

\end{proof}

\begin{remark}
A thorough study of all the maximal $\bm{Z}$-orders in $\bm{Q}\pi$
(when $\pi=D_n$ or $Q_{4n}$) can be found in \cite[pages
75-79]{Sw4}.
\end{remark}

\bigskip
Let $h_m$ be the class number of $\bm{Q}(\zeta_m)$, $h_m^+$ be the class number
of $\bm{Q}(\zeta_m+\zeta_m^{-1})$.
It is known that $h_m^+$ is a divisor of $h_m$ \cite[page 40, Theorem 4.14]{Wa}.
Recall the definition of $\pi$-tori in Definition \ref{d1.1}.

%------------------------t6.6
\begin{theorem} \label{t6.6}
Let $p$ be an odd prime number, $c\ge 1$, and $K/k$ be a Galois
extension with $\fn{Gal}(K/k) = D_{p^c}$. Then all the
$D_{p^c}$-tori defined over $k$ are stably rational over $k$ if
and only if $h_{p^c}^+=1$.
\end{theorem}

\begin{proof}
Apply Theorem \ref{t6.13}.

For any $c'\le c$, the extension $\bm{Z}[\zeta_{p^{c'}}
+\zeta_{p^{c'}}^{-1}]\hookrightarrow
\bm{Z}[\zeta_{p^c}+\zeta_{p^c}^{-1}]$ has one fully ramified prime
divisor. Thus $h_{p^c}^+=1$ implies $h_{p^{c'}}^+=1$ for any
$c'\le c$ \cite[page 39, Proposition 4.11]{Wa}.
\end{proof}

\begin{remark}
(1) The case of $D_p$-tori in the above theorem is proved by
Hoshi, Kang and Yamasaki by a different method \cite[Theorem
1.5]{HKY}.

(2) According to Washington \cite[page 420]{Wa}, the calculation
of $h_m^+$ is rather sophisticated. It is known that $h_m^+=1$ if
$m\le 66$; if the generalized Riemann hypothesis is assumed, then
$h_m^+=1$ if $m\le 161$ \cite[page 421]{Wa}. By \cite{Mi},
$h_{2^t}^+=1$ if $2^t=128, 256$; so is $2^t=512$ if the
generalized Riemann hypothesis is assumed.
\end{remark}

\bigskip

We turn to the situation of $2$-groups such as $D_n$ (the dihedral
group of order $2n$ with $n \ge 2$) and $Q_{4n}$ (the generalized
quaternion group of order $4n$ with $n \ge 2$).

The following proposition is an easy consequence of Endo and
Miyata's Theorems in \cite{EM2}, \cite{EM4} and \cite{EM6}. We
record it just to keep the reader aware.

\begin{prop} \label{p6.8}
Let $\pi=D_n$, the dihedral group of order $2n$ where $n=2^t$ and
$t \ge 1$. Then $C(\Omega_{\bm{Z}\pi}) \simeq T^g(\pi)$.
Consequently, if $h_{2^t}^+=1$ \rm{(}e.g. $1 \le t \le 8$\rm{)}
and $M$ is a $\pi$-lattice, then $M$ is both flabby and coflabby
if and only if it is stably permutation, i.e. $M \oplus P_1 \simeq
P_2$ where $P_1$ and $P_2$ are permutation lattices.
\end{prop}

\begin{proof}
Since $\pi$ is a dihedral group, $C^q
(\bm{Z}\pi)=\widetilde{C}(\bm{Z}\pi)$ by \cite[page 709, Theorem
4.6]{EM2}. Thus $C(\Omega_{\bm{Z}\pi}) \simeq C(\bm{Z}\pi)/C^q
(\bm{Z}\pi)$. On the other hand, $\pi$ is a $2$-group, it follows
that $C(\bm{Z}\pi)/C^q (\bm{Z}\pi) \simeq T^g(\pi)$ by Theorem
\ref{t2.14}. Hence the result.

Suppose that $h_{2^t}^+=1$. Then $h_{2^s}^+=1$ for all $1 \le s
\le t$ by the same arguments as in the proof of Theorem
\ref{t6.6}. It follows that $C(\Omega_{\bm{Z}\pi})=0$ according to
Theorem \ref{t6.13}. Hence $T^g(\pi)=0$.

The condition $T^g(\pi)=0$ is equivalent to $[M]^{fl}=0$ for any
invertible $\pi$-lattice $M$, which means that there is a short
exact sequence $0 \to M \to P_1 \to P_2 \to 0$ for some
permutation lattices $P_1$ and $P_2$. But this sequence splits
because of Lemma \ref{l2.6}; thus $M$ is stably permutation. By
Theorem \ref{t2.15}, $M$ is stably permutation if and only if it
is flabby and coflabby.

Note that, by \cite[page 421; Mi]{Wa}, $h_{2^t}^+=1$ if $1 \le t
\le 8$.
\end{proof}

\medskip

\begin{lemma} \label{l6.16}
Let $\pi=Q_{4n}$ be the generalized quaternion group of order $4n$
where $n=2^t$ with $t \ge 1$. If $\Omega_{\bm{Z}\pi}$ is a maximal
order in $\bm{Q}\pi$ containing $\bm{Z}\pi$, then
$C(\Omega_{\bm{Z}\pi}) \simeq \oplus_{d \mid 2n} C(\bm{Z}[\zeta_d
+ \zeta_d^{-1}]$. In particular, if $\pi=Q_8, Q_{16}, Q_{32},
Q_{64}$ or $Q_{128}$, then $C(\Omega_{\bm{Z}\pi})=0$.
\end{lemma}

\begin{proof}
Apply Theorem \ref{t6.13}. It suffices to determine $C_A(R)$ where
$R=\bm{Z}[\zeta_{2n}+\zeta_{2n}^{-1}]$,
$K=\bm{Q}(\zeta_{2n}+\zeta_{2n}^{-1})$ is the quotient field of
$R$, and $A$ is the central simple $K$-algebra defined by $A=L +
Lu$ with $u^2=-1$, $u \alpha = \tau(\alpha) u$ for any $\alpha \in
L$ ($\tau$ acts on $L=\bm{Q}(\zeta_{2n})$ by
$\tau(\zeta_{2n})=\zeta_{2n}^{-1}$).

Use the exact sequence $U(R) \to D \to P(R)/P_A(R) \to 0$ in
Definition \ref{d6.11}. By Weber's Theorem \cite[Satz 6, page 29;
CR2, page 272]{Ha} the map $U(R) \to D$ is surjective. Thus
$P_A(R)=P(R)$ and $C_A(R) \simeq C(R) \simeq
C(\bm{Z}[\zeta_{2n}+\zeta_{2n}^{-1}])$.

By \cite[page 421]{Wa}, $h_{2^s}^{+}=1$ if $1 \le s \le 6$. Hence
$C(\Omega_{\bm{Z}\pi})=0$ if $\pi=Q_8, Q_{16}, Q_{32}, Q_{64}$ or
$Q_{128}$.
\end{proof}

\bigskip
\begin{prop} \label{p6.9}
If $\pi \simeq Q_8, Q_{16}, Q_{32}, Q_{64}$ or $Q_{128}$, then
$C(\Omega_{\bm{Z}\pi})=0 = T^g(\pi)$. It follows that an
invertible $\pi$-lattice is always stably permutation.
\end{prop}

\begin{proof}
Since there is a surjection $C(\Omega_{\bm{Z}\pi})\to
C(\bm{Z}\pi)/C^q(\bm{Z}\pi)$ (see Definition \ref{d2.12}), it
follows that $C(\bm{Z}\pi)/C^q(\bm{Z}\pi)=0$ by Lemma \ref{l6.16}.

Because of Theorem \ref{t2.14}, $T^g(\pi)\simeq
C(\bm{Z}\pi)/C^q(\bm{Z}\pi)$. Thus $T^g(\pi)=0$. The remaining
proof is similar to that of Proposition \ref{p6.8}.
\end{proof}

\bigskip
The same argument of the above proposition may be applied to the
semi-dihedral groups and the modular groups also. Let $n=2^t$
where $t \ge 3$, define $SD_{2n}=\langle\sigma,\tau:
\sigma^n=\tau^2=1,\tau^{-1}\sigma\tau=\sigma^{-1+(n/2)}\rangle$
(the semi-dihedral group of order $2n$), and define
$M_{2n}=\langle\sigma,\tau:
\sigma^n=\tau^2=1,\tau^{-1}\sigma\tau=\sigma^{1+(n/2)}\rangle$
(the modular group of order $2n$). If $\pi$ is a $2$-group of
order $\ge 16$ and contains a cyclic normal subgroup of index $2$,
then $\pi$ is isomorphic to the dihedral group, the semi-dihedral
group, the generalized quaternion group or the modular group \cite[page 107]{Su}.

The proof of the following proposition is similar to that of
Theorem \ref{t6.13}, and is omitted.

\begin{prop} \label{p6.14}
Let $n=2^t$ where $t \ge 3$.

(1) If $\pi = SD_{2n}$, then $C(\Omega_{\bm{Z}\pi})\simeq
(\oplus_{0\le s\le t-1} C(\bm{Z}[\zeta_{2^s}+\zeta_{2^s}^{-1}]))
\oplus C(\bm{Z}[\zeta_n-\zeta_n^{-1}])$.

(2) If $\pi = M_{2n}$, then $C(\Omega_{\bm{Z}\pi})\simeq
(\oplus_{0\le s\le t-1} C(\bm{Z}[\zeta_{2^s}])^{(2)}) \oplus
C(\bm{Z}[\zeta_{n/2}])$.
\end{prop}

\bigskip
\begin{prop} \label{p6.15}
If $\pi \simeq M_{16}, M_{32}$ or $M_{64}$, then
$C(\Omega_{\bm{Z}\pi})=0 = T^g(\pi)$. It follows that an
invertible $\pi$-lattice is always stably permutation.
\end{prop}

\begin{proof}
By Proposition \ref{p6.14}, $C(\Omega_{\bm{Z}\pi})=0$ because of
\cite[page 205, Theorem 11.1]{Wa}. Thus $T^g(\pi)=0$.
\end{proof}

\begin{remark}
The situation of the group $SD_{2n}$ is left open because we don't
know the class number of $\bm{Q}(\zeta_n-\zeta_n^{-1})$ when $n
\ge 16$.

\end{remark}

\newpage
%----------------------------------------References

\end{document}